\documentclass[a4paper]{amsart}
\usepackage[utf8]{inputenc}
\usepackage[T1]{fontenc}
\usepackage{microtype}

\usepackage[font=small]{subcaption}	
\usepackage{float}
\usepackage{pdflscape}
\usepackage{booktabs}
\usepackage{longtable}
\usepackage{multirow}
\usepackage{multicol}
\usepackage{makecell}
\usepackage{cellspace}
\usepackage{enumitem}

\usepackage{setspace}
\usepackage{csvsimple}

\usepackage{graphicx}
\usepackage{caption}
\usepackage{pgffor,pgfmath} 
\usepackage{tikz}
    \usetikzlibrary{calc,matrix,arrows,shapes}
    \tikzset{node distance=2cm, auto}
\usepackage[all]{xy}
\newcolumntype{C}{>{$}c<{$}}
\newcolumntype{L}{>{$}l<{$}}
\newcolumntype{R}{>{$}r<{$}}

\usepackage{pifont}
\newcommand{\cmark}{\ding{51}}
\newcommand{\xmark}{\ding{55}}

\usepackage{amsmath,amssymb}
\usepackage{amsthm,mathrsfs}
\usepackage{mathtools}

\usepackage{hyperref}

\usepackage[capitalise]{cleveref}

\theoremstyle{plain}
\newtheorem{theorem}{Theorem}[section]
\newtheorem{lemma}[theorem]{Lemma}
\newtheorem{proposition}[theorem]{Proposition}
\newtheorem{corollary}[theorem]{Corollary}

\theoremstyle{definition}

\newtheorem{remark}[theorem]{Remark}

\newtheorem{construction}[theorem]{Construction}
\newtheorem{setting}[theorem]{Setting}

\Crefname{construction}{Construction}{Constructions}

\def\CC{{\mathbb C}}
\def\KK{{\mathbb K}}
\def\TT{{\mathbb T}}
\def\ZZ{{\mathbb Z}}

\def\QQ{{\mathbb Q}}
\def\PP{{\mathbb P}}

\def\KKK{\mathcal{K}}

\def\RRR{{\mathcal R}}

\DeclareMathOperator{\Eff}{Eff}
\DeclareMathOperator{\Mov}{Mov}
\DeclareMathOperator{\SAmple}{SAmple}
\DeclareMathOperator{\Ample}{Ample}
\DeclareMathOperator{\rk}{rank}

\def\quot{/\!\!/}

\def\bangle#1{{\langle #1 \rangle}}

\DeclareMathOperator{\rank}{rank}

\DeclareMathOperator{\Cl}{Cl}

\def\WDiv{{\rm WDiv}}

\def\Pic{{\rm Pic}}

\def\Spec{{\rm Spec}}

\def\cone{{\rm cone}}
\def\conv{{\rm conv}}

\def\reg{{\rm reg}}

\def\rlv{{\rm rlv}}

\newcommand{\sstable}{\mathrm{ss}}

\DeclareMathOperator{\Cox}{\mathcal{R}}
\DeclareMathOperator{\LP}{LP}

\newcommand{\free}[1]{#1^0}
\newcommand{\tor}[1]{#1^{\mathrm{tor}}}

\usepackage[backrefs,lite]{amsrefs}

\title[On smooth Calabi-Yau threefolds of Picard number two]%
{On smooth Calabi-Yau threefolds\\ of Picard number two}

\author[C.~Mauz]{Christian Mauz}

\address{Mathematisches Institut, Universit\"at T\"ubingen,
Auf der Morgenstelle 10, 72076 T\"ubingen, Germany}
\email{mauz@math.uni-tuebingen.de}


\begin{document}

\begin{abstract}
We classify all smooth Calabi-Yau threefolds of Picard number two
that have a general hypersurface Cox ring.
\end{abstract}

\maketitle

\section{Introduction}
This article contributes to the explicit classification
of Calabi-Yau threefolds.
Recall that a Calabi-Yau variety is
a normal projective complex variety~$X$
with trivial canonical class $\KKK_X$,
at most canonical singularities
and $h^i(X, \mathcal{O}_X) = 0$ for $i  = 1, \dotsc, \dim(X)-1$.
Calabi-Yau varieties
form a vast and actively studied area of research,
also aiming for classification results
such as \cites{KrSk, CDL, GHL, Og, Og94} or more
recently \cites{SchSk,Cy,GRvdH,HaKa}.

The present paper takes up the classification approach
based on positively graded rings \cites{BrGe,BrKaLe,BrKa}
yet in the multigraded setting.
We study Calabi-Yau threefolds $X$ in terms of their Cox ring.
Recall that the Cox ring of a normal
projective variety $X$ with finitely generated divisor class group
$\Cl(X) = K$ is the graded algebra
\[
	\mathcal{R}(X) \coloneqq \bigoplus_{[D] \in \Cl(X)} \Gamma(X, \mathcal{O}_X(D)).
\]
We consider the case that our Calabi-Yau threefold
$X$ comes with a hypersurface Cox rings, that means
that we have a $K$-graded presentation
\[
	\mathcal{R}(X)
	= R_g
	= \CC[T_1, \dotsc, T_r] / \langle g \rangle
\]
with a homogeneous polynomial $g$ of degree $\mu \in K$
such that $T_1, \dotsc, T_r$ form a minimal system of $K$-prime
generators for $R_g$.
In particular $\mathcal{R}(X) = R_g$ is a finitely generated $\CC$-algebra,
hence $X$ is a Mori dream space in the sense of \cite{HuKe}.
Note that a smooth Calabi-Yau variety of dimension at most three is
a Mori dream space if and only if
its cone of effective divisors is rational polyhedral~\cite{McK}.
More general, Mori dream spaces of Calabi-Yau type
are completely characterized via the singularities
of their total coordinate space $\Spec\,\Cox(X)$~\cite{KaOk}.

Following \cite{HLM},
we say that $R_g$ resp.\ $g$ is \emph{spread}
if each monomial of degree~$\mu$ is a 
convex combination over those monomials showing up in $g$ with
non-zero coefficient.
Besides, we call $R_g$
\emph{general (smooth, Calabi-Yau)} if $g$ 
\index{hypersurface Cox ring!Calabi-Yau}
admits an open neighbourhood~$U$
in the finite dimensional vector space of all
$\mu$-homogeneous polynomials
such that every $h \in U$ yields a
hypersurface Cox ring~$R_{h}$ of a
normal (smooth, Calabi-Yau) variety~$X_h$ with
divisor class group~$K$.
In~\cite{HLM} general hypersurface Cox rings were applied to
the classification of smooth Fano fourfolds
of Picard number two.

In dimension two Calabi-Yau varieties are K3 surfaces.
Their Cox rings have been studied in~\cites{AHL,ACL,Ot},
in particular describing several classes of K3 surfaces
with a hypersurface Cox ring.
Numbers~1, 2, 6 and~12 from
Oguiso's classification of smooth Calabi-Yau threefolds
that arise as a general complete intersection
in some weighted projective space~\cite{Og}
comprise all smooth Calabi-Yau threefolds $X$ with $\Pic(X) = \ZZ$ 
having a general hypersuface Cox ring;
see also \cite{IF}.
Besides, Przyjalkowski and Shramov have established
explicit bounds for smooth Calabi-Yau weighted complete intersections
in any dimension~\cite{PrSh}.

\medskip
Our main result concerns smooth Calabi-Yau threefolds
of Picard number two with a hypersurface Cox ring $R_g$.
Any projective variety $X$ with class group $K$ and
 Cox ring $R_g$ is encoded
by $R_g$ and an ample class $u \in K$
in the sense that~$X$ occurs
as the GIT quotient of the set 
of $u$-semistable points 
of $\Spec \, R_g$ by the quasitorus 
$\Spec \, \CC[K]$. 
In this setting, we write $w_i = \deg(T_i)$ and refer to 
the Cox ring generator degrees
$w_1,  \ldots, w_r \in K$, the 
relation degree $\mu \in K$ and an ample class $u \in K$ as 
\emph{specifying data} \index{specifying data}
 of the variety $X$.
Note $r = \dim(X) + 1 + \rank(K)$, hence
a hypersurface Cox ring $R_g$ of a threefold with Picard number two
has six generators $w_1, \dotsc, w_6$ .

\begin{theorem}
\label{thm:CY3folds}
The following table lists specifying data,
$w_1, \dotsc, w_6$, $\mu$ and~$u$ in~$\Cl(X)$ for all
smooth Calabi-Yau threefolds~$X$ of Picard number two
that have a spread hypersurface Cox ring.

\begin{center}
\scriptsize
\setlength\arraycolsep{2pt}
\setlength\tabcolsep{2pt}
\newcounter{CYthreefold}
\newcommand{\mycolumnwidth}{.475\textwidth}

\begin{minipage}[t]{\mycolumnwidth}
\centering
\begin{longtable}[b]{>{\refstepcounter{CYthreefold}\theCYthreefold}LCCCC}
\toprule
\multicolumn{1}{c}{No.} & \Cl(X) & [w_1, \dotsc, w_6] & \mu & u \\ \midrule

\label{CY3fold:I} &
\ZZ^2 &
\begin{bmatrix*}[r]
	1 & 1 & 1 & 0 & 0 & 0 \\
	0 & 0 & 0 & 1 & 1 & 1
\end{bmatrix*} &
\begin{bmatrix} 3 \\  3 \end{bmatrix} & \begin{bmatrix} 1 \\  1 \end{bmatrix} \\ \midrule

\label{CY3fold:I-tors1} &
\ZZ^2 \times \ZZ / 3 \ZZ &
\begin{bmatrix*}[r]
	1 & 1 & 1 & 0 & 0 & 0 \\
	0 & 0 & 0 & 1 & 1 & 1 \\
	\bar{0} & \bar{1} & \bar{2} & \bar{0} & \bar{1} & \bar{2}
\end{bmatrix*} &
\begin{bmatrix} 3 \\  3 \\ \bar{0} \end{bmatrix} & \begin{bmatrix} 1 \\  1 \\ \bar{0} \end{bmatrix} \\ \midrule

\label{CY3fold:II-1} &
\ZZ^2 &
\begin{bmatrix*}[r]
	1 & 1 & 1 & 1 & 0 & 0 \\
	0 & 0 & 1 & 1 & 1 & 1
\end{bmatrix*} &
\begin{bmatrix} 4 \\  4 \end{bmatrix} & \begin{bmatrix} 2 \\  1 \end{bmatrix} \\ \midrule

\label{CY3fold:II-2} &
\ZZ^2 &
\begin{bmatrix*}[r]
	1 & 1 & 1 & 3 & 0 & 0 \\
	0 & 0 & 1 & 3 & 1 & 1
\end{bmatrix*} &
\begin{bmatrix} 6 \\  6 \end{bmatrix} & \begin{bmatrix} 2 \\  1 \end{bmatrix} \\ \midrule

\label{CY3fold:III-i-1} &
\ZZ^2 &
\begin{bmatrix*}[r]
	1 & 1 & 1 & 0 & 0 & -1 \\
	0 & 0 & 0 & 1 & 1 & 1
\end{bmatrix*} &
\begin{bmatrix} 2 \\  3 \end{bmatrix} & \begin{bmatrix} 1 \\  1 \end{bmatrix} \\ \midrule

\label{CY3fold:III-i-2} &
\ZZ^2 &
\begin{bmatrix*}[r]
	1 & 1 & 1 & 0 & 0 & -2 \\
	0 & 0 & 0 & 1 & 1 & 1
\end{bmatrix*} &
\begin{bmatrix} 1 \\  3 \end{bmatrix} & \begin{bmatrix} 1 \\  1 \end{bmatrix} \\ \midrule

\label{CY3fold:III-i-3-tors} &
\ZZ^2 \times \ZZ / 3 \ZZ &
\begin{bmatrix*}[r]
	1 & 1 & 1 & 0 & 0 & -3 \\
	0 & 0 & 0 & 1 & 1 & 1 \\
	\overline{0} & \overline{1} & \overline{2} &
	\overline{1} & \overline{2} & \overline{0}
\end{bmatrix*} &
\begin{bmatrix} 0 \\  3 \\ \overline{0} \end{bmatrix} & \begin{bmatrix} 1 \\  1 \\ \overline{0} \end{bmatrix} \\ \midrule

\label{CY3fold:III-i-5} &
\ZZ^2 &
\begin{bmatrix*}[r]
	1 & 1 & 1 & 0 & 0 & -3 \\
	0 & 0 & 0 & 2 & 3 & 1
\end{bmatrix*} &
\begin{bmatrix} 0 \\  6 \end{bmatrix} & \begin{bmatrix} 1 \\  1 \end{bmatrix} \\ \midrule

\label{CY3fold:III-ii-1} &
\ZZ^2 &
\begin{bmatrix*}[r]
 	 1 & 1 & 1 & 1 & 0 & 0 \\
	-2 & 0 & 0 & 0 & 1 & 1
\end{bmatrix*} &
\begin{bmatrix} 4 \\  0 \end{bmatrix} & \begin{bmatrix} 1 \\  1 \end{bmatrix} \\ \midrule

\label{CY3fold:III-ii-2} &
\ZZ^2 &
\begin{bmatrix*}[r]
 	 1 & 1 & 1 & 3 & 0 & 0 \\
	-2 & 0 & 0 & 0 & 1 & 1
\end{bmatrix*} &
\begin{bmatrix} 6 \\  0 \end{bmatrix} & \begin{bmatrix} 1 \\  1 \end{bmatrix} \\ \midrule

\label{CY3fold:III-iii-1a} &
\ZZ^2 & 
\begin{bmatrix*}[r]
 	1 & 1 & 1 & 1 & 0 & 0 \\
	0 & 0 & 0 & 1 & 1 & 1
\end{bmatrix*} &
\begin{bmatrix} 4 \\  3 \end{bmatrix} & \begin{bmatrix} 2 \\  1 \end{bmatrix} \\ \midrule

\label{CY3fold:III-iii-1b} & \ZZ^2 & 
\begin{bmatrix*}[r]
 	1 & 1 & 1 & 1 & 0 & 0 \\
	0 & 0 & 0 & 1 & 1 & 1
\end{bmatrix*} &
\begin{bmatrix} 4 \\  3 \end{bmatrix} & \begin{bmatrix} 1 \\ 2 \end{bmatrix}\\ \midrule

\label{CY3fold:III-iii-2} &
\ZZ^2 &
\begin{bmatrix*}[r]
 	1 & 1 & 1 & 3 & 0 & 0 \\
	0 & 0 & 0 & 2 & 1 & 1
\end{bmatrix*} &
\begin{bmatrix} 6 \\  4 \end{bmatrix} & \begin{bmatrix} 1 \\  1 \end{bmatrix} \\ \midrule

\label{CY3fold:IV-ii-1} &
\ZZ^2 &
\begin{bmatrix*}[r]
 	 1 & 1 & 1 & 3 & 0 & 0 \\
	-1 & 0 & 0 & 1 & 1 & 1
\end{bmatrix*} &
\begin{bmatrix} 6 \\  2 \end{bmatrix} & \begin{bmatrix} 1 \\  1 \end{bmatrix} \\ 

\bottomrule

\end{longtable}
\end{minipage}
\begin{minipage}[t]{\mycolumnwidth}
\centering
\begin{longtable}[t]{>{\refstepcounter{CYthreefold}\theCYthreefold}LCCCC}
\toprule
\multicolumn{1}{c}{No.} & \Cl(X) & [w_1, \dotsc, w_6] & \mu & u \\ \midrule

\label{CY3fold:IV-ii-2a} &
\ZZ^2 & 
\begin{bmatrix*}[r]
 	 1 & 1 & 1 & 1 & 0 & 0 \\
	-1 & 0 & 0 & 1 & 1 & 1
\end{bmatrix*} &
\begin{bmatrix} 4 \\  2 \end{bmatrix} & \begin{bmatrix} 2 \\  1 \end{bmatrix} \\ \midrule

\label{CY3fold:IV-ii-2b} & 
\ZZ^2 & 
\begin{bmatrix*}[r]
 	 1 & 1 & 1 & 1 & 0 & 0 \\
	-1 & 0 & 0 & 1 & 1 & 1
\end{bmatrix*} &
\begin{bmatrix} 4 \\  2 \end{bmatrix} & \begin{bmatrix} 1 \\  2 \end{bmatrix}\\ \midrule

\label{CY3fold:IV-ii-3a} &
\ZZ^2 & 
\begin{bmatrix*}[r]
 	 1 & 1 & 1 & 1 & 0 & 0 \\
	-2 & 0 & 0 & 1 & 1 & 1
\end{bmatrix*} &
\begin{bmatrix} 4 \\  1 \end{bmatrix} & \begin{bmatrix} 2 \\  1 \end{bmatrix} \\ \midrule

\label{CY3fold:IV-ii-3b} & 
\ZZ^2 & 
\begin{bmatrix*}[r]
 	 1 & 1 & 1 & 1 & 0 & 0 \\
	-2 & 0 & 0 & 1 & 1 & 1
\end{bmatrix*} &
\begin{bmatrix} 4 \\  1 \end{bmatrix} & \begin{bmatrix} 1 \\  2 \end{bmatrix}\\ \midrule

\label{CY3fold:V-i-1} &
\ZZ^2 &
\begin{bmatrix*}[r]
 	1 & 1 & 1 & 2 & 1 & 0 \\
	0 & 0 & 0 & 1 & 1 & 1
\end{bmatrix*} &
\begin{bmatrix} 6 \\  3 \end{bmatrix} & \begin{bmatrix} 3 \\  1 \end{bmatrix} \\ \midrule

\label{CY3fold:V-i-2} &
\ZZ^2 &
\begin{bmatrix*}[r]
 	1 & 1 & 1 & 4 & 1 & 0 \\
	0 & 0 & 0 & 2 & 1 & 1
\end{bmatrix*} &
\begin{bmatrix} 8 \\  4 \end{bmatrix} & \begin{bmatrix} 3 \\  1 \end{bmatrix} \\ \midrule

\label{CY3fold:V-i-3} &
\ZZ^2 &
\begin{bmatrix*}[r]
 	1 & 1 & 1 & 5 & 2 & 0 \\
	0 & 0 & 0 & 2 & 1 & 1
\end{bmatrix*} &
\begin{bmatrix} 10 \\  4 \end{bmatrix} & \begin{bmatrix} 3 \\  1 \end{bmatrix} \\ \midrule

\label{CY3fold:V-i-4} &
\ZZ^2 &
\begin{bmatrix*}[r]
 	1 & 1 & 2 & 5 & 1 & 0 \\
	0 & 0 & 0 & 2 & 1 & 1
\end{bmatrix*} &
\begin{bmatrix} 10 \\  4 \end{bmatrix} & \begin{bmatrix} 3 \\  1 \end{bmatrix} \\ \midrule

\label{CY3fold:V-i-5} &
\ZZ^2 &
\begin{bmatrix*}[r]
 	1 & 1 & 2 & 7 & 3 & 0 \\
	0 & 0 & 0 & 2 & 1 & 1
\end{bmatrix*} &
\begin{bmatrix} 14 \\  4 \end{bmatrix} & \begin{bmatrix} 4 \\  1 \end{bmatrix} \\ \midrule

\label{CY3fold:V-ii-1} &
\ZZ^2 &
\begin{bmatrix*}[r]
 	 1 & 1 & 1 & 2 & 1 & 0 \\
	-2 & 0 & 0 & 0 & 1 & 1
\end{bmatrix*} &
\begin{bmatrix} 6 \\  0 \end{bmatrix} & \begin{bmatrix} 2 \\  1 \end{bmatrix} \\ \midrule

\label{CY3fold:V-ii-2} &
\ZZ^2 &
\begin{bmatrix*}[r]
 	 1 & 1 & 1 & 1 & 3 & 0 \\
	-2 & 0 & 0 & 0 & 1 & 1
\end{bmatrix*} &
\begin{bmatrix} 7 \\  0 \end{bmatrix} & \begin{bmatrix} 4 \\  1 \end{bmatrix} \\ \midrule

\label{CY3fold:V-ii-3} &
\ZZ^2 &
\begin{bmatrix*}[r]
 	 2 & 1 & 1 & 1 & 3 & 0 \\
	-2 & 0 & 0 & 0 & 1 & 1
\end{bmatrix*} &
\begin{bmatrix} 8 \\  0 \end{bmatrix} & \begin{bmatrix} 4 \\  1 \end{bmatrix} \\ \midrule

\label{CY3fold:VI-i-1a} &
\ZZ^2 & 
\begin{bmatrix*}[r]
 	1 & 1 & 2 & 5 & 1 & 0 \\
	0 & 0 & 1 & 3 & 1 & 1
\end{bmatrix*} &
\begin{bmatrix} 10 \\  6 \end{bmatrix} & \begin{bmatrix} 3 \\  1 \end{bmatrix} \\ \midrule

\label{CY3fold:VI-i-1b} & 
\ZZ^2 & 
\begin{bmatrix*}[r]
 	1 & 1 & 2 & 5 & 1 & 0 \\
	0 & 0 & 1 & 3 & 1 & 1
\end{bmatrix*} &
\begin{bmatrix} 10 \\  6 \end{bmatrix} & \begin{bmatrix} 3 \\  2 \end{bmatrix}\\ \midrule

\label{CY3fold:VI-ii-1} &
\ZZ^2 &
\begin{bmatrix*}[r]
 	 1 & 1 & 1 & 4 & 1 & 0 \\
	-1 & 0 & 0 & 1 & 1 & 1
\end{bmatrix*} &
\begin{bmatrix} 8 \\  2 \end{bmatrix} & \begin{bmatrix} 5 \\  1 \end{bmatrix} \\ \midrule

\label{CY3fold:VI-iii-1} &
\ZZ^2 &
\begin{bmatrix*}[r]
 	 1 & 2 & 1 & 1 & 1 & 0 \\
	-1 & -1 & 0 & 0 & 1 & 1
\end{bmatrix*} &
\begin{bmatrix} 6 \\  0 \end{bmatrix} & \begin{bmatrix} 2 \\  1 \end{bmatrix} \\ 

\bottomrule
\end{longtable}
\end{minipage}
\end{center}
\bigskip

\noindent
Moreover, each of the items~1 to~\ref{CY3fold:VI-iii-1}
even defines a general smooth Calabi-Yau hypersurface Cox ring and
thus provides the specifying data for a whole family of smooth
Calabi-Yau threefolds.
Any two smooth Calabi-Yau threefolds of Picard number
two with specifying data from distinct items of the table
are not isomorphic to each other.
\end{theorem}

Note that the varieties from \cref{thm:CY3folds} constitute a
finite number of families. 
For recent general results on boundedness of
Calabi-Yau threefolds we refer to \cite{CDHJS} as well as
\cite{Wi21} for the case of Picard number two.

\medskip
Hypersurfaces in toric Fano varieties
form a rich source of examples for Calabi-Yau varieties,
e.g.\ \cites{Ba94,Ba17,ACG16}.
\Cref{thm:CY3folds} comprises several varieties of this type.

\begin{remark}
Any Mori dream space $X$ can be embedded into a projective
toric variety by choosing  
a graded presentation of its Cox ring $\Cox(X)$; see
\cite{ADHL}*{Sec. 3.2.5} for details.
The following table shows for which varieties $X$ from
\cref{thm:CY3folds} the presentation $\Cox(X) = R_g$ gives rise
to an embedding into a (possibly singular) toric Fano variety.
Observe that in our situation this simply means
$\mu \in \Ample(X)$.

\begin{center}
\begin{tabular}{ccccccccccccccc}
1 & 2 & 3 & 4 & 5 & 6 & 7 & 8 & 9 & 10 & 11 & 12 & 13 & 14 & 15 \\ \hline
	\cmark & \cmark & \xmark & \xmark & \cmark & \cmark & 
	\xmark &\xmark &\xmark &\xmark &
	\cmark &
	\xmark &
	\cmark & \cmark & \cmark \\[\medskipamount]
16 & 17 & 18 & 19 & 20 & 21 & 22 & 23 & 24 & 25 & 26 & 27 & 28 & 29 & 30  \\ \hline
	\xmark &
	\cmark &
	\xmark &
	\cmark &\cmark &\cmark &\cmark &\cmark &
	\xmark & \xmark & \xmark & 
	\xmark & \cmark & \cmark & \xmark  
\end{tabular}
\end{center}

\end{remark}

\tableofcontents

The author would like to thank Jürgen Hausen for
his constant support and very valuable advice.
In addition, the author is grateful to 
Antonio Laface for helpful disucssions
and his interest in this work.

\section{Background on Mori dream spaces}
\label{sec:toolboxCY}
A Mori dream space
is an irreducible normal projective
variety $X$ with finitely generated divisor class group $\Cl(X)$
and finitely generated Cox ring $\mathcal{R}(X)$.
As observed by Hu and Keel \cite{HuKe}, 
Mori dream spaces are characterized by their eponymous feature,
which is optimal behavior with respect to the Mori program~\cites{MMP1,MMP2}.
In this section we gather basic facts
on the combinatorial description of Mori dream spaces
from \cite{ADHL}.
The ground field $\KK$ is algebraically closed and of
characteristic zero.

\medskip
\noindent

Let us first recall some terminology concerning graded algebras.
Consider a finitely generated abelian group $K$
and an integral $\KK$-algebra $R = \bigoplus_{w \in K} R_w$
with a $K$-grading.     
A non-zero non-unit $f \in R$ is \emph{$K$-prime}
if it is homogeneous and $f \mid gh$ with
homogeneous elements $g, h \in R$ implies $f \mid g$ or 
$f \mid h$.
We say that $R$ is \emph{$K$-factorial}
or that the $K$-grading on $R$ is \emph{factorial}
if any homogeneous non-zero non-unit
is a product of $K$-prime elements.
Fix a system $f_1, \dotsc, f_r \in R$ of pairwise non-associated
$K$-prime generators for $R$. 
The \emph{effective cone} and the \emph{moving cone}
of $R$ in the rational vector space 
$K_\QQ = \QQ \otimes_\ZZ  K$ associated 
with~$K$ are
\[
\Eff(R) 
\coloneqq
\cone(\deg(f_1), \dotsc, \deg(f_r)), \qquad
\Mov(R)
\coloneqq
\bigcap_{i=1}^r
\cone(\deg(f_j);\; j \neq i).
\]
This definition does not depend on the choice of
$f_1, \dotsc, f_r$.
The $K$-grading on $R$ is called \emph{pointed}
if $R_0 = \KK$ holds and $\Eff(R)$ contains no line
and it is called \emph{almost free} if any $r-1$ of
$\deg(f_1), \dotsc, \deg(f_r)$ generate $K$ as a group.

Moreover, by an \emph{abstract Cox ring}
we mean a $K$-graded algebra $R$ such that
\begin{enumerate}
\item
$R$ is normal, integral and finitely generated,

\item 
$R$ has only constant homogeneous 
units, the $K$-grading is almost free, 
pointed and factorial, and

\item
the moving cone $\Mov(R)$ 
is of full dimension in $K_\QQ$.
\end{enumerate}

The Cox ring of a Mori dream space always
satisfies the conditions of an abstract Cox ring. 
Vice versa, we can produce Mori dream spaces
from abstract Cox rings
using the following construction ~\cite{ADHL}*{Constr. 3.2.1.3}.

\begin{construction} 
\label{constr:mdsCY}
Let $R$ be an abstract Cox ring and 
consider the action of the quasitorus 
$H = \Spec \, \KK[K]$ on the affine 
variety $\bar X  = \Spec \, R$.
For every GIT-cone $\lambda \in \Lambda(R)$ 
with
$\lambda^\circ \subseteq \Mov(R)^\circ$, 
we set
\[
X(\lambda) 
\ := \ 
\bar X^{ss}(\lambda) \quot H .
\]
Then $X = X(\lambda)$ is normal, projective 
and of dimension $\dim(R)-\dim(K_\QQ)$.
The divisor class group and 
the Cox ring of $X$ are given as
$$ 
\Cl(X) \ = \ K,
\qquad\qquad
\mathcal{R}(X) 
\ = \ 
\bigoplus_{\Cl(X)} \Gamma(X,\mathcal{O}_X(D))
\ = \ 
\bigoplus_{K} R_w 
\ = \ 
R.
$$
Moreover, the cones of effective, 
movable, semiample and ample 
divisor classes of~$X$ are given 
in $\Cl_\QQ(X) = K_\QQ$ as 
$$ 
\Eff(X) \ = \ \Eff(R), 
\qquad
\Mov(X) \ = \  \Mov(R), 
$$
$$
\SAmple(X) = \lambda,
\qquad
\Ample(X) = \lambda^\circ.
$$
\end{construction}
See \cite{ADHL}*{Thm. 3.2.1.4} for the description of the Cox ring
and \cite{ADHL}*{Prop. 3.3.2.9} for the description of the cones
of effective, movable, semiample and ample divisors.
Moreover \cite{ADHL}*{Thm. 3.2.1.9} guarantees that
indeed all Mori dream spaces arise from \cref{constr:mdsCY} 

\medskip
Choosing homogeneous generators for an abstract Cox ring
gives rise to a closed embedding
into a projective toric variety~\cite{ADHL}*{Constr. 3.2.5.7}.

\begin{construction}
\label{ambtv}
In the situation of \cref{constr:mdsCY},
consider a graded presentation
\[
    R = \mathbb{K}[T_1, \dotsc, T_r] / \mathfrak{a} 
\]
where $T_1, \dotsc, T_r$ define pairwise non-associated $K$-primes in $R$
and $\mathfrak{a} \subseteq S = \mathbb{K}[T_1, \dotsc, T_r]$ is
a homogeneous ideal. 
The GIT-fan $\Lambda(S)$ w.r.t.\ the diagonal $H$-action on $\KK^r = \Spec\,S$
refines the GIT-fan $\Lambda(R)$. Let $\tau \in \Lambda(S)$
with $\lambda^\circ \subseteq \tau^\circ$.  
Running \cref{constr:mdsCY} for $S$ and $\tau$
yields a projective toric variety $Z$
fitting in the following diagram
\[ \xymatrix{
    \bar X^\sstable (\lambda) \ar[r] \ar[d]_{\quot H} &
    (\KK^r)^\sstable (\tau) \ar[d]^{\quot H} \\
    X \ar[r]^\imath & Z
}\]
The embedding $\imath: X \rightarrow Z$ is neat, i.e.\
it is a closed embedding,
the torus invariant prime divisors on $Z$ restrict to pairwise
different prime divisors on $X$ and the induced pullback of
divisor class groups $\imath^*: \Cl(Z) \rightarrow \Cl(X)$
is an isomorphism.
\end{construction}

Varieties arising from abstract Cox rings admit
the following description for their local behavior 
\citelist{\cite{ADHL}*{Cor. 3.3.1.9} \cite{ADHL}*{Cor 3.3.1.12}}.

\begin{proposition}
\label{locprops}
In the situation of \cref{ambtv} the following hold.
\begin{enumerate}
\item
$X$ is $\QQ$-factorial if and only if $\lambda$ is of full dimension.

\item
$X$ is smooth if and only if 
$\bar X^\sstable$ is smooth and $X \subseteq Z^{\reg}$ holds.
\end{enumerate}
\end{proposition}

Furthermore, for hypersurface Cox rings
we have an explicit formula
for the anticanonical class~\cite{ADHL}*{Prop. 3.3.3.2}.

\begin{proposition} 
\label{prop:anticanclassCY}
Consder the situation of
Construction~\ref{ambtv}.
If $\mathfrak{a} = \langle g \rangle$ 
holds, then the anticanonical class of $X$ is
given in $K = \Cl(X)$ as 
\[
-\KKK_X
\ = \ 
\deg(T_1) +  \dotsb + \deg(T_r)
-
\deg(g).
\]
\end{proposition}

We call an irreducible normal variety $X$ \emph{weakly Calabi-Yau}
if its canonical class~$\KKK_X$ vanishes.
For varieties with hypersuface Cox ring
this notion only depends on the generator degrees
and the relation degree.
Moreover, it turns out that smooth weakly Calabi-Yau hypersurfaces
are Calabi-Yau varieties in the strong sense.

\begin{remark}
\label{rem:weaklyCY}
In the situation of
Construction~\ref{ambtv}
assume $\mathfrak{a} = \langle g \rangle$.
\begin{enumerate}
\item
From \cref{prop:anticanclassCY} we deduce that $X$ is weakly Calabi-Yau
if and only if $\mu = w_1 + \dotsb + w_r$ holds.
In particular, $\mu$ lies in the relative interior of $\Eff(R)$
whenever $X$ is weakly Calabi-Yau.

\item
If $X$ is weakly Calabi-Yau,
then \cref{prop:anticanclassCY}
shows that $X$ is an anticanonical hypersurface of a
projective toric variety $Z$ as in \cref{ambtv}.
If, in addition, $X$ is smooth, then \cref{locprops} allows us
to apply \cite{ACG}*{Prop. 6.1}.
From this we infer $h^i(X, \mathcal{O}_X) = 0$
for all $0 < i < \dim(X)$,
hence~$X$ is Calabi-Yau.
\end{enumerate}
\end{remark}

Let us briefly recall the notion of flops~\cite{Ko,KoMo}
in our situation.
A Weil divisor~$D$ on a variety $X$ is said to be
\emph{relatively ample} w.r.t\ a morphism $\varphi: X \rightarrow Y$
of varieties, or just \emph{$\varphi$-ample}, if
there is an open affine covering $Y = \bigcup V_i$
such that~$D$ restricts to an ample divisor
on each $\varphi^{-1}(V_i)$.
Moreover a proper birational morphism $\varphi: X \rightarrow Y$
of normal varieties is called \emph{extremal}, if
$X$ is $\QQ$-factorial and for each two Cartier divisors $D_1, D_2$ on $X$
there are $a_1, a_2 \in \ZZ$ where at least one of $a_1, a_2$ 
is non-zero and $a_1 D_1 - a_2 D_2$
is linearly equivalent to the pullback $\varphi^* C$ of some
Cartier divisor $C$ on $Y$.
A birational map $\psi: X^- \dashrightarrow X^+$ 
of $\QQ$-factorial weakly Calabi-Yau varieties
is a \emph{flop} if it fits 
into a commutative diagram
\[ \xymatrix{
X^- \ar@{-->}[rr]^\psi \ar[rd]_{\varphi^-} & & X^+ \ar[ld]^{\varphi^+} \\
& Y & 
}\]
where $\varphi^-: X^- \rightarrow Y$ and
$\varphi^+: X^- \rightarrow Y$ 
are small proper birational morphims,
$\varphi^-$ is extremal and there is a Weil divisor~$D$ on~$X^-$
such that $-D$ is $\varphi^-$-ample
and the proper transform of $D$ on $X^+$ is $\varphi^+$-ample.

\medskip
Later we will use that
weakly Calabi-Yau
Mori dream spaces of Picard number two
that share a common Cox ring
are connected by flops;
for convenience we give a direct proof here.

\begin{proposition}
\label{floplemma}
Let $R$ be an abstract Cox ring with grading group $K$ 
of rank two and
$\lambda, \eta \in \Lambda(R)$ full-dimensional cones with
$\lambda^\circ, \eta^\circ \subseteq \Mov(R)^\circ$.
Consider the varieties $X(\lambda)$ and $X(\eta)$
arising from \cref{constr:mdsCY}.
If the canonical class of $X(\lambda)$ is trival, then
there is a sequence of flops
\[
    X(\lambda) \dashrightarrow X_1
    \dashrightarrow \dotsb \dashrightarrow
    X_k \dashrightarrow X(\eta).
\]
\end{proposition}

We study the toric setting first. 
Consider $S = \KK[T_1, \dotsc, T_r]$
with a linear, pointed, almost free grading
of an abelian group~$K$ of rank two
and the associated action of the quasitorus $H = \Spec\, \KK[K]$ on $\KK^r$.
Let us recall some facts about
toric varieties arising from GIT-cones
as treated e.g.\ in \cite{ADHL}*{Chap. 2--3}.
The degree homomorphism
$Q: \ZZ^r \rightarrow K$, $e_i \mapsto w_i \coloneqq \deg(T_i)$
gives rise to a pair 
of mutually dual exact sequences:
\[ \xymatrix@C=4em{
0 \ar[r] & L \ar[r] & \ZZ^r \ar[r]^{P} & \ZZ^n \\
0 & K \ar[l] & \ar[l]_Q \ZZ^r  & \ar[l]_{P^*} \ZZ^n & \ar[l] 0
}
\]
Given a GIT-cone $\tau \in \Lambda(S)$ with $\tau^\circ \subseteq \Mov(S)^\circ$,
the associated toric variety $Z = (\KK^r)^\sstable (\tau) \quot H$ 
has the describing fan $\Sigma(\tau)$ given by
\[
    \Sigma(\tau) = \{ P(\gamma_0^*);\; \gamma_0 \in \rlv(\tau) \},
    \qquad
    \rlv(\tau) =
    \{ \gamma_0 \preceq \gamma;\, \tau^\circ \subseteq Q(\gamma_0)^\circ \}
\]
In particular all such fans share the same one-skeleton
consisting of the pairwise different rays 
generated by $v_1, \dotsc, v_r$ where $v_i \coloneqq P(e_i) \in \ZZ^n$.
Moreover, we denote~$Z_{\gamma_0}$ for the affine toric variety
associated with the lattice cone $P(\gamma_0^*) \subseteq \QQ^n$.
The covering of $Z$ by affine toric charts then formulates as 
\[
    Z = \bigcup_{\gamma_0 \in \rlv(\tau)} Z_{\gamma_0}.
\]

\begin{lemma}
\label{galefiberformula}
Let $\tau_1, \tau_2 \in \Lambda(S)$ 
with $\tau_i^\circ \subseteq \Mov(S)^\circ$.
Then for any $\gamma_1 \in \rlv(\tau_1)$, $\gamma_2 \in \rlv(\tau_2)$
we have
\[
    P(\gamma_2^*) \subseteq P(\gamma_1^*)
    \Longleftrightarrow
    \gamma_1 \subseteq \gamma_2.
\]
\end{lemma}

\begin{proof}
The implication ``$\Leftarrow$'' is clear.
We show ``$\Rightarrow$''.
Note that the cones $P(\gamma_1^*) \in \Sigma(\tau_1)$
and $P(\gamma_2^*) \in \Sigma(\tau_2)$
both live in lattice fans having precisely $v_1, \dotsc, v_r$
as primitive ray generators. 
Thus for $j = 1, 2$ and any $v_i$ we have
\[
    v_i \in P(\gamma_j^*)
    \; \Longleftrightarrow \; \QQ_{\geq 0}\, v_i
    \text{ is an extremal ray of } P(\gamma_j^*)
    \; \Longleftrightarrow \;
    e_i \in \gamma_j^*.
\]
From this we infer that $P(\gamma_2^*) \subseteq P(\gamma_1^*)$
implies $\gamma_2^* \subseteq \gamma_1^*$.
This in turn means $\gamma_1 \subseteq \gamma_2$.
\end{proof}

Let $\tau^-, \tau^+ \subseteq \QQ^2 = K_\QQ$ be full-dimensional GIT-cones with
$(\tau^-)^\circ, (\tau^+)^\circ \subseteq \Mov(S)^\circ$
intersecting in a common ray $\tau^0 \coloneqq \tau^- \cap \tau^+$.
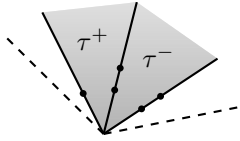
\begin{center}
\begin{tikzpicture}[scale=0.6]
\coordinate (u1) at (5,1);
\coordinate (u2) at (3,2);
\coordinate (u3) at (1,4);
\coordinate (u4) at (-1,2);
\coordinate (u5) at (-1,1);

\foreach \i in {1,...,5}
{
\path let \p1=(u\i), \n1={veclen(\x1,\y1)} in
	coordinate (v\i) at (${1/\n1}*(u\i)$);
}

\foreach \v[remember=\v as \lastv (initially v2)] in {v2,v3,v4}
	\path[top color=gray!25!, bottom color=gray!60!] 
		(0,0)--($3cm*(\lastv)$)--($3cm*(\v)$)--cycle;

\foreach \v/\dmax in {v2/2,v3/2,v4/1}
{
	\draw[thick,color=black] (0,0) -- ($3cm*(\v)$);
	\foreach \d in {1,...,\dmax}
		\fill (${1cm*((1+(\d-1)*0.5)}*(\v)$) circle (0.5ex);
}

\draw[thick,dashed,color=black] (0,0) -- ($3cm*(v1)$);
\draw[thick,dashed,color=black] (0,0) -- ($3cm*(v5)$);

\draw[color=black] ($2.25cm*0.5*(v2) + 2.25cm*0.5*(v3)$) node {$\tau^-$};
\draw[color=black] ($2.25cm*0.5*(v3) + 2.25cm*0.5*(v4)$) node {$\tau^+$};
\end{tikzpicture}   
\end{center}
Consider the projective toric varieties $Z^0$, $Z^-$, $Z^+$
associated with $\tau^0$, $\tau^-$ and $\tau^+$
and denote $\Sigma^0 = \Sigma(\tau^0)$, $\Sigma^- = \Sigma(\tau^-)$
and $\Sigma^+ = \Sigma(\tau^+)$ for the describing fans.
Moreover the inclusions of the respective semistable points
induce proper birational toric morphims
$\varphi^-: Z^- \rightarrow Z^0$, $\varphi^+: Z^+ \rightarrow Z^0$
described by the refinements of fans $\Sigma^- \preceq \Sigma^0$
and $\Sigma^+ \preceq \Sigma^0$ respectively.
This yields a small birational map $\psi: Z^- \dashrightarrow Z^+$
as shown in the diagram
\[ \xymatrix{
(\KK^r)^{\mathrm{ss}}(\tau^-) \ar@{}[r]|\subseteq \ar[d]_{\quot H} &
(\KK^r)^{\mathrm{ss}}(\tau^0) \ar[d]_{\quot H} &
(\KK^r)^{\mathrm{ss}}(\tau^+) \ar[d]_{\quot H} \ar@{}[l]|\supseteq \\
Z^- \ar[r]^{\varphi^-} \ar@/_1.2em/@{-->}[rr]_\psi &
Z^0 &
Z^+ \ar[l]_{\varphi^+}
} \]

\begin{lemma}
\label{toricflop}
Let $-D$ be an ample divisor on $Z^-$, then 
$D$ regarded as a divisor on~$Z^+$ is $\varphi^+$-ample.
\end{lemma}

\begin{proof}
By suitably applying an automorphism of $K$
and relabeling $w_1, \dotsc, w_r \in K$
we achieve counter-clockwise ordering i.e.\
\[
    i \leq j \Longrightarrow \det(w_i, w_j) \geq 0
\]
and $\det(w^-, w^+) \geq 0$ for all $w^- \in \tau^-$,
$w^+ \in \tau^+$.
Moreover, we name the indices of the weights
that approximate $\tau^0$ from the outside
\[
    i^- \coloneqq \max(i; w_i \in \tau^-), \qquad
    i^+ \coloneqq \min(i; w_i \in \tau^+).
\]
The geometric constellation of $w_1, \dotsc, w_r$ in $\QQ^2$
directly yields that the set of minimal cones of $\rlv(\tau^0)$
is
\[
\{ \gamma_i;\; i^- < i < i^+ \}
\cup \{ \gamma_{i,j};\; i \leq i^-,\, j \geq i^+ \},
\]
where
$\gamma_{i_1, \dotsc, i_r} =
\cone(e_{i_1}, \dotsc, e_{i_r}) \preceq \gamma$.
The corresponding cones $P(\gamma_0)^*$ are precisely
the maximal cones of $\Sigma^0$, in particular
the associated toric charts $Z_{\gamma_0}$
form an open affine covering of $Z^0$.
We show that $D$ is ample on each open subset
$(\varphi^+)^{-1}(Z_{\gamma_0})$ of $Z^+$.

First, note that $\varphi^+$ is an isomorphism over the affine toric charts
of $Z^0$ associated with the common minimal cones of $\rlv(\tau^0)$
and $\rlv(\tau^+)$, namely all $Z_{\gamma_{i,j}}$
where $i \leq i^+$ and $j \geq i^+$.
In particular each preimage
$(\varphi^+)^{-1}(Z_{\gamma_{i,j}})$ is affine.
Since $Z^+$ is $\QQ$-factorial by \cref{locprops},
the divisor $D$ is $\QQ$-Cartier thus
restricts to an ample divisor on any open affine subvariety of~$Z^+$.

It remains to consider 
the charts of $Z^0$ defined by
the faces of the form~$\gamma_j$.
Let us fix some $i^- < j < i^+$.
The minimal cones $\gamma_0 \in \rlv(\tau^+)$
with $\gamma_j \subseteq \gamma_0$
are precisely those of the form $\gamma_{j,i}$
where $i \geq j$.
As the toric morphism $\varphi^+$
is described by the refinement $\Sigma^+ \preceq \Sigma^0$,
\cref{galefiberformula} yields
\[
    U \coloneqq (\varphi^+)^{-1}\left( Z_{\gamma_j} \right)
    = \bigcup_{i \geq i^+}  Z_{\gamma_{j,i}} \subseteq Z^+.
\]
Note that $U \subseteq Z^+$ is an open toric subset and
the maximal cones of the associated subfan~$\Sigma'$ of~$\Sigma^+$ 
are precisely the cones $P(\gamma_{j,i}^*)$ where $i \geq i^+$.
This shows that the rays of $\Sigma'$ are the rays of $\Sigma^+$ minus
$\varrho_j$.
Thus the divisor class group of $U$ is given by
$\Cl(U) = K / \langle w_j \rangle$ and the projection
corresponds to the restriction of divisor classes
\[ \xymatrix{
    \Cl(Z^+) \ar[r]^{\imath^*} \ar[d]_{\cong} & \Cl(U) \ar[d]^{\cong} \\
    K \ar[r] &  K / \langle w_j \rangle
} \]
Taking $\rk K = 2$ into account, we may choose suitable coordinates leading to
an isomorphism $\Cl(U)_\QQ \cong \QQ$
such that for any $w \in \Cl(Z^+)$ the restriction $\imath^*(w)$ to $\Cl(U)$
and $\det(w_j, w)$ have the same sign.
Graphically this means that the sign of
$\imath^*(w) \in \Cl(U)$ is positive 
if~$w$ lies above the ray~$\tau^0$ and negative if~$w$ lies
below~$\tau^0$.

Since we know the maximal cones of $\Sigma'$ we may compute the ample cone
of $U$ as 
\[
    \Ample(U)
    = \bigcap_{i \geq i^+} \left(\imath^* \circ Q(\gamma_{j,i})\right)^\circ
    = \QQ_{> 0} \subseteq \QQ = \Cl(U)_\QQ.
\]
Note that $[-D] \in \Ample(Z^-) = \tau^-$ lies below $\tau$, thus
the class of $-D$ (regarded on $Z^+)$ restricted to $U$ is negative,
hence
$\imath^*[D] \in \Ample(U)$.
In other words, $D$ is ample on $U$.
Altogether, we conclude that $D$ is $\varphi^+$-ample.
\end{proof}

\begin{proof}[Proof of \cref{floplemma}]
First, we deal with the case that $\lambda$ and $\eta$ intersect in a common
ray $\varrho \coloneqq \lambda \cap \eta$. Consider a $K$-graded presentation
\[
    R = \mathbb{K}[T_1, \dotsc, T_r] / \mathfrak{a}
\]
where $T_1, \dotsc, T_r$ define pairwise non-associated $K$-primes in $R$
and $\mathfrak{a} \subseteq S = \mathbb{K}[T_1, \dotsc, T_r]$ is
a homogeneous ideal.
The GIT-fan $\Lambda(S)$ w.r.t.\ the $H$-action on $S$
refines the GIT-fan $\Lambda(R)$. We may choose $\tau^+, \tau^- \in \Lambda(S)$
such that
\[
    \left(\tau^-\right)^\circ \subseteq \lambda^\circ  \quad
    \left(\tau^+\right)^\circ \subseteq \eta^\circ, \quad
   \tau^- \cap \tau^+ = \varrho.
\]
The toric morphisms $\varphi_Z^-$, $\varphi^+_Z$ arising from the face relations
$\varrho \preceq \tau^-, \tau^+$ of GIT-cones are compatible with
the toric morphisms $\varphi^-$, $\varphi^+$ arising from $\varrho \preceq \lambda, \eta$
as shown in the following diagram
where the vertical arrows are neat embeddings as in \cref{ambtv}
\[ \xymatrix{
Z(\tau^-) \ar[r]^{\varphi_Z^-} & 
Z(\varrho) &
Z(\tau^+) \ar[l]_{\varphi_Z^+}\\
X(\lambda) \ar[u] \ar[r]^{\varphi^-} & 
X(\varrho) \ar[u] &
X(\eta) \ar[u] \ar[l]_{\varphi^+}
} \]
We claim that the resulting birational map
$\psi: X(\lambda) \dashrightarrow X(\eta)$ is a flop.
First observe that
$X^-$ is $\QQ$ factorial by \cref{locprops}~(i)
and  $\varphi^-, \varphi^+$ are small birational morphisms;
see~\cite{ADHL}*{Rem. 3.3.3.4}

We show that $\varphi^-$ is extremal. 
Let $D_1, D_2$ be Cartier divisors on $X$.
If $D_1, D_2$ lie on a common ray in $\Pic(X)_\QQ = \QQ^2$
we find $a_1, a_2 \in \ZZ$, not both zero, such that $a_1 D_1 - a_2 D_2$ is
principal,
thus linearly equivalent to the pullback of any prinicipal divisor on $Y$.
Otherwise the subgroup $G \subseteq \Pic(X)$
spanned by $[D_1], [D_2] \in \Pic(X)$ is of rank two
and thus of finite index in $\Pic(X)$.
Hence, for any Cartier divisor $C$ on $Y$, we find $b \in \ZZ_{> 0}$
such that $b [(\varphi^-)^*(C)] \in G$, i.e., 
$a_1 [D_1] + a_2[D_2] = b [(\varphi^-)^*(C)]$
for some $a_1, a_2 \in \ZZ$.
In other words, $a_1 D_1 - a_2 D_2$ is linearly equivalent
to the pullback of $b C$.
If $C$ is not principal, then $(\varphi^-)^* C$ is not principal either,
thus at least one of $a_1, a_2$ is non-zero.

Let $D_Z$ be a torus invariant divisor on $Z(\tau^-)$
such that $-D_Z$ is ample for $Z(\tau^-)$.
Since $X(\lambda) \subseteq Z(\tau^-)$ is neatly embedded,
we may restrict $D_Z$ to a divisor $D_X$ on~$X(\lambda)$.
Note that $-D_X$ is ample since $-D_Z$ is so.
In particular $-D_X$ is $\varphi^-$-ample.
\Cref{toricflop} yields that~$D_Z$ is $\varphi_Z^+$-ample.
Let $U \subseteq Z(\varrho)$ be an affine open subset such that
$D_Z$ is ample on
\[
    V \coloneqq \left(\varphi_Z^+ \right)^{-1}(U) \subseteq Z(\tau^+).
\]
The further restriction of $D_Z$ from $V$
to $V \cap X(\eta)$ is still ample.
In other words,~$D_X$ restricted
to $(\varphi^+)^{-1}(X(\tau) \cap U)$ is ample.
We conclude that $D_X$ is $\varphi^+$-ample.

Altogether $\psi: X(\lambda) \dashrightarrow X(\eta)$
is a flop.

In the general case we find full-dimensional GIT-cones
$\lambda = \eta_1, \dotsc, \eta_k = \eta$
where $\eta_i^\circ \subseteq \Mov(R)^\circ$ holds for all $i$
and each intersection $\eta_i \cap \eta_{i+1}$ is a ray of $\Lambda(R)$.
According to the preceeding discussion, we may successively construct
the desired sequence of flops.
\end{proof}

\section{Combinatorial constraints on smooth hypersurface Cox rings}
The proof of \cref{thm:CY3folds}
basically uses the combinatorial framework
for the classification of smooth Mori dream spaces of Picard number two
with hypersurface Cox ring
established in \cite{HLM}*{Sec. 5}; see also \cite{Ma}.
Let us recall the notation from there and slightly extend it to
address the torsion subgroup of the grading group explicitly.
We also present the accompanying toolkit.
Moreover we add some new tools for dealing with torsion.

We work over an algebraically closed field $\KK$ of characteristic zero.

\begin{setting}
\label{setting:rho2forCY}
Consider $K = \ZZ^2 \times \Gamma$
where $\Gamma$ is some finite abelian group of order~$t$,
a $K$-graded algebra $R$
and $X = X(\lambda)$,
where $\lambda \in \Lambda(R)$ with 
$\lambda^\circ \subseteq \Mov(R)^\circ$,
as in Construction~\ref{constr:mdsCY}.
Assume that we have an irredundant 
$K$-graded presentation
$$ 
R 
\ = \ 
R_g 
\ = \ 
\KK[T_1, \ldots, T_r] / \bangle{g}
$$
\index{$R_g$}
such that the $T_i$ define pairwise 
nonassociated $K$-primes in $R$.
Write $w_i \coloneqq \deg(T_i)$,
$\mu \coloneqq \deg(g)$ for the degrees 
in $K$.
According to the presentation $K = \ZZ^2 \times \Gamma$ we denote
\[
	w_i = (u_i, \zeta_i), \quad
	\mu = (\alpha, \theta), \qquad
	u_i, \alpha \in \ZZ^2,\quad
	\zeta_i, \theta \in \Gamma.
\]
Similarly the degree matrix $Q = [w_1, \dotsc, w_r]$
is divided into a free part $Q^0$ and a torsion part $\tor{Q}$,
i.e., we set
\[
	\free{Q} = 
	\begin{bmatrix}
	u_1 & \hdots & u_r
	\end{bmatrix}, \quad
	\tor{Q} = 
	\begin{bmatrix}
	\zeta_1 & \hdots & \zeta_r
	\end{bmatrix}.
\]
Regarded as elements of $K_\QQ$ we identify $w_i$ with $u_i$
and $\mu$ with $\alpha$.
Suitably numbering $w_1, \ldots, w_r$,
we ensure counter-clockwise ordering,
that means that we always have
\[
i \le j \ \implies \ \det(w_i,w_j) = \det(u_i, u_j) \ge 0.
\]
Note that each ray of $\Lambda(R)$ is of 
the form $\varrho_i  = \cone(w_i)$, but not 
vice versa.
We assume $X$ to be $\QQ$-factorial.
According to Proposition~\ref{locprops} 
this means $\dim(\lambda) = 2$.
Then the  effective cone of $X$ is uniquely 
decomposed into three convex sets,
$$
\Eff(X) 
\ = \ 
\lambda^- \cup \lambda^\circ \cup \lambda^+,
$$
where~$\lambda^-$ and~$\lambda^+$ are convex polyhedral 
cones not intersecting
$\lambda^\circ = \Ample(X)$ and 
$\lambda^- \cap \lambda^+$ consists of the origin.
\begin{center}
    \begin{tikzpicture}[scale=0.6]
    \path[fill=gray!60!] (0,0)--(3.5,2.9)--(0.6,3.4)--(0,0);
    \path[fill, color=black] (1.75,1.45) circle (0.5ex)  node[]{};
    \path[fill, color=black] (1.4,2.4) circle (0.0ex)  node[]{\small{$\lambda^\circ$}};
    \path[fill, color=black] (0.3,1.7) circle (0.5ex)  node[]{};
    \draw (0,0)--(0.6,3.4);
    \draw (0,0) --(-2,3.4);
  \path[fill, color=black] (-1,1.7) circle (0.5ex)  node[left]{\small{$w_r$}};
    \path[fill, color=black] (-0.35,2.65) circle (0.0ex)  node[]{\small{$\lambda^+$}};
    \draw (0,0)  -- (3.5,2.9);
    \draw (0,0)  -- (3.5,0.5);
  \path[fill, color=black] (1.75,0.25) circle (0.5ex)  node[below]{\small{$w_1$}};
    \path[fill, color=black] (2.6,1.2) circle (0.0ex)  node[]{\small{$\lambda^-$}};
    \path[fill, color=white] (4,1.9) circle (0.0ex);
  \end{tikzpicture}   
\end{center}
\end{setting}

\begin{remark} 
Setting~\ref{setting:rho2forCY}
is respected by orientation preserving 
automorphisms of~$K$.
If we apply an orientation reversing 
automorphism of $K$, 
then we regain Setting~\ref{setting:rho2forCY}
by reversing the numeration of 
$w_1, \ldots, w_r$.
Moreover, we may interchange the numeration 
of $T_i$ and $T_j$ if $w_i$ and $w_j$
share a common ray without 
affecting \cref{setting:rho2forCY}.
We call these operations \emph{admissible
coordinate changes}. 
\index{admissible coordinate change}
Note that any automorphism of $\ZZ^2$ naturally extends
to an automorphism of $K = \ZZ^2 \times \Gamma$
acting as the identity on $\Gamma$.
\end{remark}

We state an adapted version of \cite{HLM}*{Prop. 2.4}
locating the relation degree.

\begin{proposition}
\label{prop:hypmovforCY}
In the situation of \cref{setting:rho2forCY} we have
$\mu \in \cone(w_3, w_{r-2}) \subseteq K_\QQ$.
\end{proposition}

A further important observation is that
the GIT-fan structure of $R_g$ can be read of
from the geometric constellation of
$w_1, \dotsc, w_r$ and $\mu$.

\begin{proposition} \label{cor:gitconeCY}
Situation as in \cref{setting:rho2forCY}.
Assume that $X(\lambda)$ is locally factorial
and $R$ is a spread hypersurface Cox ring.
Then the full-dimensional cones of $\Lambda(R)$
are precisely the cones $\eta = \cone(w_i, w_j)$
where $\varrho_i \neq \varrho_j$
and one of the following conditions is satisfied:

\begin{enumerate}
\item
$\mu \in \varrho_i$ holds, $\varrho_i$ contains at least
two generator degrees and
$\eta^\circ$ contains no generator degree,

\item
$\mu \in \varrho_j$ holds, $\varrho_j$ contains at least
two generator degrees and $\eta^\circ$
contains no generator degree,

\item
$\mu \in \eta^\circ$ holds and there is at most one $w_k \in \eta^\circ$,
which lays on the ray through $\mu$,

\item
$\mu \notin \eta$ holds and $\eta^\circ$ contains no generator degrees.
\end{enumerate}
\end{proposition}

Recall that a point $x \in X$ of a variety $X$
is \emph{factorial} if its stalk $\mathcal{O}_{X,x}$
admits unique factorization. We call $X$ \emph{locally factorial}
if every point $x \in X$ is factorial.
In particular smooth varieties are locally factorial.

The following lemmas are crucial in
gaining constraints on specifying data
of hypersurface Cox rings.

\begin{lemma} 
\label{lem:twofacesCY} 
Situation as in \cref{setting:rho2forCY}.
Let $i, j$ with $\lambda \subseteq \cone(w_i,w_j)$.
If $X = X(\lambda)$ is locally factorial,
then either $w_i,w_j$ generate $K$ as a group,
or $g$ has precisely one monomial of the form 
$T_i^{l_i} T_j^{l_j}$, where $l_i+l_j > 0$.
\end{lemma}

\begin{lemma} 
\label{lem:threegenerateCY}
Let $X = X(\lambda)$ be as in Setting~\ref{setting:rho2forCY}
and let $1 \leq i < j < k \leq r$.
If $X$ is locally factorial, 
then $w_i, w_j, w_k$ generate $K$ as a group
provided  that one of the following holds:
\begin{enumerate}
\item 
$w_i, w_j \in \lambda^-$, $w_k \in \lambda^+$ and $g$ 
has no monomial of the form $T_k^{l_k}$,
\item 
$w_i \in \lambda^-$, $w_j, w_k \in \lambda^+$ and $g$ 
has no monomial of the form $T_i^{l_i}$,
\item
$w_i \in \lambda^-$, $w_j \in \lambda^\circ$, $w_k \in \lambda^+$.
\end{enumerate}

Moreover, if (iii) holds,
then $g$ has a monomial of the form $T_j^{l_j}$
where $l_j$ is divisible by the order of the
factor group $K / \langle w_i, w_k \rangle$.
In particular $l_j$ is a multiple of $\det(u_i, u_k)$.
\end{lemma}

\begin{lemma} 
\label{lem:2on1ray}
Assume $u, w_1, w_2$ generate the abelian group $\mathbb{Z}^2$.
If $w_i = a_i w$ holds with a primitive $w \in \mathbb{Z}^2$
and $a_i \in \mathbb{Z}$,
then $(u, w)$ is a basis for $\mathbb{Z}^2$ 
and $u$ is primitive. 
\end{lemma}

Now we present some structural observations
which prove useful at different places
inside the proof of \cref{thm:CY3folds}
when we deal with specific configurations of
generator and relation degrees. 

\begin{lemma} 
\label{lem:raypowerCY}
In Setting~\ref{setting:rho2forCY},
assume that $X = X(\lambda)$ is locally factorial
and $R_g$ a spread hypersurface Cox ring.
If $w_i$ lies on the ray through $\mu$,
then $g$ has a monomial of the form $T_i^{l_i}$
where $l_i \ge 2$.
\end{lemma}

\begin{lemma} 
\label{lem:combminrayCY}
In \cref{setting:rho2forCY} assume that
$\Mov(R) = \Eff(R)$ and $\mu \in \Eff(R)^\circ$
hold.
Let $\Omega$ denote the set of two-dimensional cones
$\eta \in \Lambda(R)$ with
$\eta^\circ \subseteq \Mov(R)^\circ$.
\begin{enumerate}
\item 
If $X(\eta)$ is locally factorial 
for some $\eta \in \Omega$, 
then $\Eff(R)$ is a regular cone and 
every $u_i$ on the boundary of $\Eff(R)$ 
is primitive.

\item 
If $X(\eta)$ is locally factorial for 
all $\eta \in \Omega$, 
then, for any $w_i \in \Eff(R)^\circ$, 
we have $u_i = u_1 + u_r$ or $g$ has a 
monomial of the form $T_i^{l_i}$.
\end{enumerate}
\end{lemma}

\begin{lemma} \label{lem:w4onrho2CY}
Situation as in \cref{setting:rho2forCY}.
If we have $w_2 = w_3$ and
$\mu \in \varrho_2$, then $w_4 \in \varrho_2$ holds.
\end{lemma}

\begin{proof}
Suppose $w_4 \notin \varrho_2$. Then
every monomial of $g$ not being
divisible by $T_1$ is of the form
$T_2^{l_2} T_3^{l_3}$ where $l_2 + l_3 > 0$.
Since $g$ is prime, thus not divisible by $T_1$,
at least one such monomial occurs with non-zero
coefficient in $g$.
From $w_2 = w_3$ we deduce that $g_1 \coloneqq g(0, T_2, \dotsc, T_r)$
is a classical homogeneous polynomial in $T_2$, $T_3$,
thus admits a presentation $g_1 = \ell_1 \dotsm \ell_m$
where $\ell_1, \dotsc, \ell_m$ are linear forms in~$T_2$ and~$T_3$.
Here $w_2 = w_3$ ensures that $\ell_1, \dotsc, \ell_m$ are homogeneous
w.r.t.\ the $K$-grading.
Observe $m > 1$
as the presentation of $R$ is irredundant.
We conclude that $g_1$ is not $K$-prime,
hence $T_1 \in R$ is not $K$-prime either.
A contradiction.
\end{proof}

We have to bear in mind that the divisor class group $K = \Cl(X)$
of a smooth Calabi-Yau threefold $X$ is not necessarily torsion-free.
The following lemmas show that in the case of a hypersurface Cox ring
the order of the torsion subgroup
is bounded in terms of monomials of the relation degree.
A first important constraint is that the torsion subgroup of $K$ is
cyclic.

\begin{lemma}
\label{lem:torsfree}
Situation as in \cref{setting:rho2forCY}.
If $X = X(\lambda)$ is locally factorial and $\mu \notin \lambda$,
then $K \cong \ZZ^2$ holds.
\end{lemma}

\begin{proof}
We have $\lambda = \cone(w_i, w_j)$ for some generator degrees $w_i, w_j$
lying on the boundary of $\lambda$.
Due to $\mu \notin \lambda$, there is no monomial $T_i^{l_i} T_j^{l_j}$
of degree $\mu$.
\Cref{lem:twofacesCY} yields that $K$ is generated by $w_i, w_j$.
Since $\rk(K) = 2$, this implies $K \cong \ZZ^2$.
\end{proof}

\begin{lemma}
\label{lem:rho2tors}
Situation as in \cref{setting:rho2forCY}.
If $X$ is locally factorial, then $K \cong \ZZ^2 \times \ZZ / t \ZZ$
holds.
\end{lemma}

\begin{proof}
Both $\lambda^-$ and $\lambda^-$ contain at least 
two Cox ring generator degrees.
This allows us to choose $w_i, w_j, w_k$ such that \cref{lem:threegenerateCY}
applies.
This ensures that $K$ is generated by three elements.
By \cref{setting:rho2forCY} we have $\rank(K) = 2$,
thus $K$ is as claimed.
\end{proof}

\begin{lemma}
\label{lem:twovarmonomial}
Situation as in \cref{setting:rho2forCY}.
Let $1 \leq i, j \leq n$ with
$\cone(w_i, w_j) \cap~\lambda^\circ \neq~\emptyset$.
If $X = X(\lambda)$ is locally factorial and
$\mu \in \lambda$ holds, then there is a monomial
$T_i^{l_i} T_j^{l_j}$ of degree~$\mu$ where $l_i + l_j > 0$.
\end{lemma}

\begin{proof}
Since $g$ is $\mu$-homogeneous, we are done when
$g$ has a monomial of the form~$T_i^{l_i} T_j^{l_j}$
with $l_i + l_j > 0$.

We assume that $g$ has no monomial of the form $T_i^{l_i} T_j^{l_j}$.
Then $\varrho_i$ and $\varrho_j$ both are GIT-rays,
thus none of $w_i, w_j$ lies in $\lambda^\circ$.
This forces $\lambda \subseteq \cone(w_i, w_j)$.
Then \cref{lem:twofacesCY} tells us that $w_i, w_j$ generate $K$ as a group.
Using $\mu \in \lambda \subseteq \cone(w_i, w_j)$ we deduce
that $\mu$ is an positive integral combination over $w_i, w_j$,
i.e., there exists a monomial as desired.
\end{proof}

\begin{lemma}
\label{lem:rho2torsbound}
Situation as in \cref{setting:rho2forCY}.
Let $1 \leq i, j, k \leq r$ such that
$w_i, w_j, w_k$ generate $K$ as a group,
$\det(u_i, u_j) = 1$ and
$\cone(w_i, w_j) \cap \lambda^\circ \neq \emptyset$.
If $X$ is locally factorial,
then $t \mid l_k$ holds for any monomial $T_i^{l_i} T_k^{l_k}$ of degree $\mu$.
\end{lemma}

\begin{proof}
Using $\det(u_i, u_j) = 1$ enables
us to apply a suitable admissible coordinate change
such that $\zeta_i = \zeta_j = 0$.
Moreover we may assume $\lambda \in \mu$; otherwise \cref{lem:torsfree}
yields $t = 1$ and there is nothing left to show.
This allows us to use \cref{lem:twovarmonomial}.
From this we infer that $\mu = (\alpha, \theta)$ is an integral positive
combination over $w_i$, $w_j$, thus $\theta = 0$.
Since $w_i, w_j, w_k$ generate $K$ as a group,
$\zeta_k$
is a generator for $\Gamma$.
Using $\zeta_i = 0$
we obtain $l_k \zeta_k = \theta = 0$
whenever $T_i^{l_i} T_k^{l_k}$ is of degree $\mu$.
This implies~$t \mid l_k$.
\end{proof}

\begin{lemma}
\label{lem:torsboundpower}
Situation as in \cref{setting:rho2forCY}.
Assume that $X = X(\lambda)$ is locally factorial.
If $\det(u_1, u_r) = 1$ and $\alpha = l_k u_k$ holds, then $t \mid l_k$.
\end{lemma}

\begin{proof}
\Cref{lem:threegenerateCY} yields that $w_1, w_k, w_r$ generate $K$ as a group.
Besides~$T_k^{l_k}$ is of degree~$\mu$ by \cref{lem:raypowerCY}.
Now \cref{lem:rho2torsbound} tells us $t \mid l_k$.
\end{proof}

\begin{lemma}
\label{lem:rk2threegen}
Let $w_i = (u_i, \zeta_i) \in \ZZ^2 \times \ZZ / t \ZZ$
for $1 \leq i \leq 3$.
If $u_1 = u_2$ holds and $w_1, w_2, w_3$ span $\ZZ^2 \times \ZZ / t \ZZ$ as a group,
then $\zeta_1 - \zeta_2$ is a generator for $\ZZ / t \ZZ$.
\end{lemma}

\begin{lemma}
\label{lem:tors24}
Situation as in \cref{setting:rho2forCY}.
If $X$ is locally factorial, $\det(u_1, u_r) = 1$, and 
$u_i = u_j$ holds for some $1 < i < j < r$,
then $\zeta_1 - \zeta_2$ is a generator for $\Gamma$.
In particular~$K$ is torsion-free or $t \neq 2, 4$ holds.
\end{lemma}

\begin{proof}
First note that $w_i, w_j$ share a common ray
in $K_\QQ$, thus do not lie
in the relative interior of the GIT-cone $\lambda$;
see \cref{cor:gitconeCY}. 
So we have $w_i, w_j \in \lambda^-$ or $w_i, w_j \in \lambda^+$. 
By applying an orientation reversing coordinate change
if necessary we achieve $w_i, w_j \in \lambda^-$.

We have $K = \ZZ^2 \times \ZZ / t \ZZ$; see \cref{lem:rho2tors}.
Using $\det(u_1, u_r) = 1$ enables
us to apply a suitable admissible coordinate change
such that $\zeta_1 = \zeta_r = \overline{0}$.
\Cref{rem:weaklyCY} ensures that $g$ has no monomial of the form $T_r^{l_r}$.
Hence \cref{lem:threegenerateCY} yields that both triples
$w_1, w_i, w_r$ and $w_1, w_j, w_r$ generate $K$ as a group.
In particular $\zeta_i$, $\zeta_j$ both are generators
for $\ZZ / t \ZZ$.
Moreover \cref{lem:threegenerateCY} tells us
that $w_i, w_j, w_r$ form a generating set for $K$.
\Cref{lem:rk2threegen} yields that $\zeta_i - \zeta_j$
is a generator for $\ZZ / t \ZZ$.
The proof is finished by the fact that the difference of
two generators for $\ZZ / 2 \ZZ$ resp.\ $\ZZ / 4 \ZZ$ is never
a generator for the respective group.
\end{proof}

\section{Proof of Theorem~\ref{thm:CY3folds}: Collecting Candidates}
The first and major task in the proof of \cref{thm:CY3folds} is to
show that we find specifying data 
for any given smooth Calabi-Yau threefold~$X$
with spread hypersurface Cox ring 
among the items displayed in \cref{thm:CY3folds}.
This is done by a case-by-case analysis of the geometric
constellation of the Cox ring generator degrees.

Now the ground field is $\KK = \CC$.
The sole reason for this is the reference 
involved in the proof of the following proposition.

\begin{proposition}
\label{CYSQMsmooth}
Consider the situation of \cref{setting:rho2forCY}.
If $X(\lambda)$ is a smooth weakly Calabi-Yau threefold,
then any variety $X(\eta)$ arising from 
a full-dimensional GIT-cone $\eta$
satisfying $\eta^\circ \subseteq \Mov(R)^\circ$
is smooth.
\end{proposition}

\begin{proof}
\Cref{floplemma} provides us with a sequence of flops
\[
    X(\lambda) = X_1
    \dashrightarrow \dotsb \dashrightarrow
    X_k = X(\eta).
\]
According to \cite{KoMo}*{Thm. 6.15}, see also~\cite{Ko},
flops of threefolds preserve smoothness.
So we successively obtain
smoothness for all varieties in the above sequence,
especially for~$X(\eta)$.
\end{proof}

\medskip
Given a positive integer $n$,
a sum of the form
$n_1 + \dotsb + n_k = n$
where $n_1, \dotsc, n_k \in~\ZZ_{\geq 1}$
is called an \emph{integer partition} of $n$.
If one wants to emphasize the order of the summands,
one calls such a sum an \emph{integer composition} of $n$.
For instance, $1 + 1 + 2 = 4$ and $1 + 2 + 1 = 4$
are two different integer compositions of $4$ but
they are equal as integer partitions.

\begin{remark}
\label{rem:intcompsparts}
In \cref{setting:rho2forCY} the geometric constellation of $w_1, \dotsc, w_r$
is described by an integer composition of $r$
in the following sense:
First, we take into account that some of the rays $\varrho_i = \cone(w_i)$
may coincide and label the actual rays properly.
Let $1 \leq j_1 < \dotsb < j_s \leq r$ such that
$\varrho_{j_k} \neq \varrho_{j_l}$ holds for $j_k \neq j_l$ 
and each $\varrho_i$ equals some $\varrho_{j_k}$.
Set $\sigma_k \coloneqq \varrho_{j_k}$.
We denote $N_k$ for the number of
Cox ring generator degrees $w_i$ lying on $\sigma_k$.
Then the distribution of the degrees $w_i$ on the rays $\sigma_k$
is encoded by the composition
\[
	N_1 + \dotsb + N_s = r.
\]
For example, when $r = 4$ holds, the integer compositions
$1 + 1 + 2 = 4$ and $1 + 2 + 1 = 4$
correspond to the constellations
of $w_1, \dotsc, w_4$ illustrated below.
\begin{figure}[H]
\subcaptionbox*{$1 + 1 + 2 = 4$}[0.45\linewidth]
{
\begin{tikzpicture}[scale=0.6]
\coordinate (u1) at (5,1);
\coordinate (u2) at (1,2);
\coordinate (u3) at (-2,4);

\foreach \i in {1,...,3}
{
\path let \p1=(u\i), \n1={veclen(\x1,\y1)} in
	coordinate (v\i) at (${1/\n1}*(u\i)$);
}

\foreach \v[remember=\v as \lastv (initially v1)] in {v1,v2,v3}
	\path[top color=gray!25!, bottom color=gray!60!] 
		(0,0)--($3cm*(\lastv)$)--($3cm*(\v)$)--cycle;

\foreach \v/\dmax in {v1/1,v2/1,v3/2}
{
	\draw[thick,color=black] (0,0) -- ($3cm*(\v)$);
	\foreach \d in {1,...,\dmax}
		\fill (${1cm*((1+(\d-1)*0.5)}*(\v)$) circle (0.5ex);
}

\node[right] at ($3cm*(v1)$) {$\sigma_1 = \varrho_1$};
\node[above right] at ($3cm*(v2)$) {$\sigma_2 = \varrho_2$};
\node[above] at ($3cm*(v3)$) {$\sigma_3 = \varrho_3 = \varrho_4$};

\end{tikzpicture}   
} %
\subcaptionbox*{$1 + 2 + 1 = 4$}[0.45\linewidth]
{
\begin{tikzpicture}[scale=0.6]
\coordinate (u1) at (5,1);
\coordinate (u2) at (1,2);
\coordinate (u3) at (-2,4);

\foreach \i in {1,...,3}
{
\path let \p1=(u\i), \n1={veclen(\x1,\y1)} in
	coordinate (v\i) at (${1/\n1}*(u\i)$);
}

\foreach \v[remember=\v as \lastv (initially v1)] in {v1,v2,v3}
	\path[top color=gray!25!, bottom color=gray!60!] 
		(0,0)--($3cm*(\lastv)$)--($3cm*(\v)$)--cycle;

\foreach \v/\dmax in {v1/1,v2/2,v3/1}
{
	\draw[thick,color=black] (0,0) -- ($3cm*(\v)$);
	\foreach \d in {1,...,\dmax}
		\fill (${1cm*((1+(\d-1)*0.5)}*(\v)$) circle (0.5ex);
}

\node[right] at ($3cm*(v1)$) {$\sigma_1 = \varrho_1$};
\node[above right] at ($3cm*(v2)$) {$\sigma_2 = \varrho_2 = \varrho_3$};
\node[above] at ($3cm*(v3)$) {$\sigma_3 = \varrho_4$};
\end{tikzpicture}   
} %
\end{figure}
\end{remark}

\begin{proposition}
\label{prop:constellations}
Situation as in \cref{setting:rho2forCY}.
If $X$ is a weakly Calabi-Yau threefold,
then $r = 6$ holds and the constellation of $w_1, \dotsc, w_6$
corresponds to one of the following integer partitions
$N_1 + \dotsc + N_s = 6$
in the sense of \cref{rem:intcompsparts}.
{ \small
\[ \begin{array}{cccccccc}
\toprule
& s & N_1 & N_2 & N_3 & N_4 & N_5 & N_6 \\ \midrule
\text{I} & 2 & 3 & 3 & \text{---} & \text{---} & \text{---} & \text{---} \\
\text{II} & 3 & 2 & 2 & 2 & \text{---} & \text{---} & \text{---} \\
\text{III} & 3 & 1 & 2 & 3 & \text{---} & \text{---} & \text{---} \\
\text{IV} & 4 & 1 & 1 & 2 & 2 & \text{---} & \text{---} \\
\text{V} & 4 & 1 & 1 & 1 & 3 & \text{---} & \text{---} \\
\text{VI} & 5 & 1 & 1 & 1 & 1 & 2 & \text{---} \\
\text{VII} & 6 & 1 & 1 & 1 & 1 & 1 & 1 \\
\bottomrule
\end{array} \]
}
\end{proposition}

\begin{proof}
Observe $r = \dim(X) + \dim(K_\QQ) + 1 = 6$.
The subsequent table shows all integer partitions $N_1 + \dotsb + N_s = 6$.
{ \small
\[ \begin{array}{cccccccc}
\toprule
& s & N_1 & N_2 & N_3 & N_4 & N_5 & N_6 \\ \midrule
& 1 & 6 & \text{---} & \text{---} & \text{---} & \text{---} & \text{---} \\
& 2 & 1 & 5 & \text{---} & \text{---} & \text{---} & \text{---} \\
& 2 & 2 & 4 & \text{---} & \text{---} & \text{---} & \text{---} \\
\text{I} & 2 & 3 & 3 & \text{---} & \text{---} & \text{---} & \text{---} \\
\text{II} & 3 & 2 & 2 & 2 & \text{---} & \text{---} & \text{---} \\
\text{III} & 3 & 1 & 2 & 3 & \text{---} & \text{---} & \text{---} \\
& 3 & 1 & 1 & 4 & \text{---} & \text{---} & \text{---} \\
\text{IV} & 4 & 1 & 1 & 2 & 2 & \text{---} & \text{---} \\
\text{V} & 4 & 1 & 1 & 1 & 3 & \text{---} & \text{---} \\
\text{VI} & 5 & 1 & 1 & 1 & 1 & 2 & \text{---} \\
\text{VII} & 6 & 1 & 1 & 1 & 1 & 1 & 1 \\
\bottomrule
\end{array} \] }
Our task is to show that in the situation of
\cref{setting:rho2forCY} those
partitions without roman label do not admit a composition
corresponding to the constellation of $w_1, \dotsc, w_6$
in $K_\QQ$.

Observe that in the cases $s = 1$ and $s = 2$ where $N_1 = 1$, $N_2 = 5$
the moving cone $\Mov(R)$ of $R$ must be one-dimensional; a contradiction.
From \Cref{prop:hypmovforCY} we deduce that any constellation
given by $N_1 + N_2 = 2 + 4 = 6$
forces $\mu$ to live in the boundary of $\Eff(R)$.
This contradicts \cref{rem:weaklyCY}.
Furthermore, the partition $N_1 + N_2 + N_3 = 1 + 1 + 4$
comprises precisely two compositions,
that is to say
\[
	N_1 + N_2 + N_3 = 1 + 4 + 1
	\quad \text{and} \quad
	N_1 + N_2 + N_3 = 1 + 1 +4.
\]
The first of them implies that $\Mov(R)$ is one-dimensional; a contradiction.
Considering the latter, \cref{prop:hypmovforCY} shows that
$\mu$ lies on the boundary of $\Eff(R)$;
a contradiction to \cref{rem:weaklyCY}.
\end{proof}

Throughout the proof of \cref{thm:CY3folds}
we will often encounter inequations of the following type.

\begin{remark}
\label{rem:sumprodforCY}
The following table describes the solutions of the inequation
\[
	x_1 \dotsm x_n \leq x_1 + \dotsb + x_n, \quad
	x_1, \dotsc, x_n \in \ZZ_{\geq 1}
\]
for $n = 3,4,5$ where $x_1, \dotsc, x_n$ are in ascending order.
Here, $\ast$ stands for an arbitrary positive integer.

\centerline{
\small
\begin{minipage}[b]{0.33\textwidth}
\[
\begin{array}{c|ccccc}
\toprule 
n & x_1 & x_2 & x_3 & x_4 \\ \midrule
\multirow{3}{*}{3} &   1 &   1 &   \ast & \text{---} \\
  &   1 &   2 &   2 & \text{---}  \\
  &   1 &   2 &   3 & \text{---}  \\ \midrule
\multirow{3}{*}{4} & 1 & 1 & 1 & \ast \\
 & 1 & 1 & 2 & 3 \\
 & 1 & 1 & 2 & 4 \\
\bottomrule
\end{array}
\]
\end{minipage}
\begin{minipage}[b]{0.33\textwidth}
\[
\begin{array}{c|cccccc}
\toprule 
n & x_1 & x_2 & x_3 & x_4 & x_5 \\ \midrule
\multirow{6}{*}{5} & 1 & 1 & 1 & 1 & \ast \\
 & 1 & 1 & 1 & 2 & 2 \\
 & 1 & 1 & 1 & 2 & 3 \\
 & 1 & 1 & 1 & 2 & 4 \\
 & 1 & 1 & 1 & 2 & 5 \\
 & 1 & 1 & 1 & 3 & 3\\
\bottomrule
\end{array}
\]
\end{minipage}
}
\end{remark}

\medskip
\noindent
We work in \cref{setting:rho2forCY} for the proof of \cref{thm:CY3folds}.
According to \Cref{rem:weaklyCY}~(i) it suffices to 
determine the degree matrix $Q = [w_1, \dotsc, w_6]$
in order to figure out candidates
for specifying data of~$X$
since the relation degree $\mu$ is given by
\[
	\mu = w_1 + \dotsb + w_6.
\]
When $Q$ and $\mu$ are fixed,
we cover all possibilities (up to isomorphism)
by picking an interior point
$u$ of each full-dimensional GIT-chamber $\lambda$
with $\lambda^\circ \subseteq \Mov(R)^\circ$.

Our proof of \cref{thm:CY3folds} will be split into
Parts I, \dots, VII discussing the
constellations of $w_1, \dotsc, w_6$ in the
sense of \cref{rem:intcompsparts} given
by the accordingly labeled integer partition of six from 
\cref{prop:constellations}.
In the present article we elaborate Parts I--IV and~VII.
The remaining parts are treated with similar arguments
and can be found in \cite{Ma}.

\subsection*{Part I}
We consider $3 + 3 = 6$ i.e.\
the generator degrees $w_i$
are evenly distributed on two rays $\sigma_1$, $\sigma_2$.
So $w_1, \dotsc, w_6$ lie all in the boundary of $\Eff(R)$. 
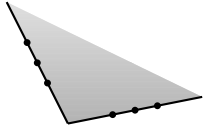
\begin{center}
\begin{tikzpicture}[scale=0.6]
\coordinate (u1) at (5,1);
\coordinate (u2) at (-2,4);

\foreach \i in {1,...,2}
{
\path let \p1=(u\i), \n1={veclen(\x1,\y1)} in
	coordinate (v\i) at (${1/\n1}*(u\i)$);
}

\foreach \v[remember=\v as \lastv (initially v1)] in {v1,v2}
	\path[top color=gray!25!, bottom color=gray!60!] 
		(0,0)--($3cm*(\lastv)$)--($3cm*(\v)$)--cycle;

\foreach \v/\dmax in {v1/3,v2/3}
{
	\draw[thick,color=black] (0,0) -- ($3cm*(\v)$);
	\foreach \d in {1,...,\dmax}
		\fill (${1cm*((1+(\d-1)*0.5)}*(\v)$) circle (0.5ex);
}

\end{tikzpicture}   
\end{center}
\Cref{lem:combminrayCY}~(i) tells us that 
each $w_i$ is primitive and $\Eff(R)$ is regular.
In particular $u_1 = u_2 = u_3$ and $u_4 = u_5 = u_6$.
A suitable admissible coordinate change leads to
\[
	Q^0 = \begin{bmatrix}
	1 & 1 & 1 & 0 & 0 & 0 \\
	0 & 0 & 0 & 1 & 1 & 1
	\end{bmatrix}.
\]
If $K$ is torsion-free, this leads to specifying data
as in Number~\ref{CY3fold:I} from \cref{thm:CY3folds}.

We assume that $K$ admits torsion.
\Cref{rem:weaklyCY}~(i) implies $\alpha = u_1 + \dotsb + u_6 = (3,3)$.
\Cref{lem:twovarmonomial} guarantees that $T_2^3 T_4^3$
is of degree $\mu$. 
\Cref{lem:threegenerateCY} tells us that
$w_1, w_2, w_4$ generate $K$ as a group
and we have $\det(u_1, u_4) = 1$.
Thus we may apply \cref{lem:rho2torsbound}.
From this we infer $t \mid 3$, 
hence $t = 3$ i.e.\
$K = \ZZ^2 \times \ZZ / 3\ZZ$; see also \cref{lem:rho2tors}.
Furthermore \cref{lem:threegenerateCY} yields
that $K$ is generated by each of the triples
\[
	(w_1, w_2, w_4), \quad
	(w_1, w_3, w_4), \quad
	(w_2, w_3, w_4).
\]
Since $u_1 = u_2 = u_3$, we conclude that
$\eta_1, \eta_2, \eta_3$ are pairwise different.
Otherwise two of $w_1, w_2, w_3$ coincide,
hence $K$ is generated by two elements; a contradiction.
In the same manner we obtain that $\eta_4, \eta_5, \eta_6$
are pairwise different.
After suitably reordering $T_1, \dotsc, T_6$ we arrive
at speciyfing data as in Number~\ref{CY3fold:I-tors1}
from \cref{thm:CY3folds}.

\subsection*{Part II}
We discuss the degree constellation
determined by $2 + 2 + 2 = 6$.
Here the generator degrees $w_i$
are evenly distributed on three rays $\sigma_1, \sigma_2, \sigma_3$.
\begin{center}
\begin{tikzpicture}[scale=0.6]
\coordinate (u1) at (5,1);
\coordinate (u2) at (1,2);
\coordinate (u3) at (-2,4);

\foreach \i in {1,...,3}
{
\path let \p1=(u\i), \n1={veclen(\x1,\y1)} in
	coordinate (v\i) at (${1/\n1}*(u\i)$);
}

\foreach \v[remember=\v as \lastv (initially v1)] in {v1,v2,v3}
	\path[top color=gray!25!, bottom color=gray!60!] 
		(0,0)--($3cm*(\lastv)$)--($3cm*(\v)$)--cycle;

\foreach \v/\dmax in {v1/2,v2/2,v3/2}
{
	\draw[thick,color=black] (0,0) -- ($3cm*(\v)$);
	\foreach \d in {1,...,\dmax}
		\fill (${1cm*((1+(\d-1)*0.5)}*(\v)$) circle (0.5ex);
}

\end{tikzpicture}   
\end{center}
We have $\mu \in \sigma_2$ by \cref{prop:hypmovforCY}.
\Cref{cor:gitconeCY} provides us with
two GIT-cones
\[
	\eta_1 = \cone(w_1, w_3), \qquad
	\eta_2 = \cone(w_3, w_5).
\]
According to \cref{CYSQMsmooth} the
associated varieties $X(\eta_1)$, $X(\eta_2)$
both are smooth.
\Cref{lem:combminrayCY}~(i) yields $u_1 = u_2$, $u_5 = u_6$
and $\det(u_1, u_5) = 1$. 
After applying a suitable admissible coordinate change
the degree matrix is of the form
\[
	Q^0 = \begin{bmatrix}
	1 & 1 & a_3 & a_4 & 0 & 0 \\
	0 & 0 & b_3 & b_4 & 1 & 1
	\end{bmatrix}, \qquad
	a_3, a_4 \in \ZZ_{\geq 1}.
\]
We may assume $a_3 \leq a_4$.
Let $v = (v_1, v_2) \in \ZZ^2$ be the primitive vector
lying on~$\sigma_2$.
Applying \cref{lem:threegenerateCY} to $X(\eta_2)$
and the triple $w_3, w_4, w_5$ shows $\gcd(a_3, a_4) = 1$.
In addition, we obtain $v_1 = 1$ from \cref{lem:2on1ray}.
\Cref{lem:threegenerateCY} again, this time applied to
$X(\eta_1)$ and $w_1, w_2, w_3$, gives $v_2 = 1$.
From $v_1 = v_2$ we deduce $a_3 = b_3$ and $a_4 = b_4$.
\Cref{lem:raypowerCY} ensures that~$\mu_1$ is divisble
by both $a_3$ and $a_4$,
thus $a_3 a_4 \mid \mu_1$.
\Cref{rem:weaklyCY}~(i) says $\mu = w_1 + \dotsc + w_6$.
We conclude
\[
	a_3 a_4 \mid \mu_1 = a_3 + a_4 + 2.
\]
First we deduce $a_4 \mid a_3 + 2$.
Moreover we obtain $a_3 \leq 4$ due to $a_3 \leq a_4$.
Altogether the integers $a_3$, $a_4$ are bounded,
so we just have to examine
the possible configurations.

\begin{itemize}[itemsep=1ex]
\item $a_3 = 1$:
From $a_4 \mid a_3 + 2 = 3$ we infer $a_4 = 1, 3$.
Now we show that $K$ is torsion-free.
For $a_4 = 1$ we have
\[
	\free{Q} = \begin{bmatrix}
	1 & 1 & 1 & 1 & 0 & 0 \\
	0 & 0 & 1 & 1 & 1 & 1
	\end{bmatrix}, \qquad
	\alpha = (4, 4).
\]
Observe $\free{\mu} = 4 u_3$.
\Cref{lem:torsfree} tells us $t \mid 4$,
thus $K$ is torsion-free according to \cref{lem:tors24}.
Similarly, for $a_4 = 3$ we have
\[
	\free{Q} = \begin{bmatrix}
	1 & 1 & 1 & 3 & 0 & 0 \\
	0 & 0 & 1 & 3 & 1 & 1
	\end{bmatrix}, \qquad
	\alpha = (6, 6).
\]
Observe $\alpha = 2 u_4$.
\Cref{lem:torsfree} tells us $t \mid 2$,
thus $K$ is torsion-free according to \cref{lem:tors24}.
We arrive at specifying data as in Numbers~\ref{CY3fold:II-1} and~\ref{CY3fold:II-2} from
\cref{thm:CY3folds}.
Observe $X(\eta_1) \cong X(\eta_2)$
in both cases due to the
symmetry of the geometric constellation of
$w_1, \dotsc, w_6, \mu$.
Thus it suffices to list an ample class
for~$X(\eta_1)$ only.

\item $a_3 = 2$:
From $a_4 \mid a_3 + 2 = 4$ and $a_3 \leq a_4$
we infer $a_4 = 2,\,4$.
This contradicts $\gcd(a_3, a_4) = 1$.

\item $a_3 = 3$:
From $a_4 \mid a_3 + 2 = 5$ and $a_3 \leq a_4$
we infer $a_4 = 5$.
This leads to $\mu_1 = a_3 + a_4 + 2 = 10$.
A contradiction to $a_3 \mid \mu_1$.

\item $a_3 = 4$:
From $a_4 \mid a_3 + 2 = 6$ and $a_3 \leq a_4$
we infer $a_4 = 6$.
This contradicts $\gcd(a_3, a_4) = 1$.
\end{itemize}

\subsection*{Part III}
In this part we consider the arrangements
of $w_1, \dotsc, w_6$ associated
with the integer partition $1 + 2 + 3 = 6$.
Here we have precisely three rays $\sigma_1, \sigma_2, \sigma_3$
each of which contains a different number of Cox ring generator degrees.
A suitable admissible coordinate change 
turns the setting into one of the following:

\begin{figure}[h]
\centering
\renewcommand\thesubfigure{\roman{subfigure}}
\subcaptionbox*{III-i}[0.3\linewidth]
{
\begin{tikzpicture}[scale=0.6]
\coordinate (u1) at (5,1);
\coordinate (u2) at (1,2);
\coordinate (u3) at (-2,4);

\foreach \i in {1,...,3}
{
\path let \p1=(u\i), \n1={veclen(\x1,\y1)} in
	coordinate (v\i) at (${1/\n1}*(u\i)$);
}

\foreach \v[remember=\v as \lastv (initially v1)] in {v1,v2,v3}
	\path[top color=gray!25!, bottom color=gray!60!] 
		(0,0)--($3cm*(\lastv)$)--($3cm*(\v)$)--cycle;

\foreach \v/\dmax in {v1/3,v2/2,v3/1}
{
	\draw[thick,color=black] (0,0) -- ($3cm*(\v)$);
	\foreach \d in {1,...,\dmax}
		\fill (${1cm*((1+(\d-1)*0.5)}*(\v)$) circle (0.5ex);
}

\end{tikzpicture}   
} %
\subcaptionbox*{III-ii}[0.3\linewidth]
{
\begin{tikzpicture}[scale=0.6]
\coordinate (u1) at (5,1);
\coordinate (u2) at (1,2);
\coordinate (u3) at (-2,4);

\foreach \i in {1,...,3}
{
\path let \p1=(u\i), \n1={veclen(\x1,\y1)} in
	coordinate (v\i) at (${1/\n1}*(u\i)$);
}

\foreach \v[remember=\v as \lastv (initially v1)] in {v1,v2,v3}
	\path[top color=gray!25!, bottom color=gray!60!] 
		(0,0)--($3cm*(\lastv)$)--($3cm*(\v)$)--cycle;

\foreach \v/\dmax in {v1/1,v2/3,v3/2}
{
	\draw[thick,color=black] (0,0) -- ($3cm*(\v)$);
	\foreach \d in {1,...,\dmax}
		\fill (${1cm*((1+(\d-1)*0.5)}*(\v)$) circle (0.5ex);
}

\end{tikzpicture}   
} %
\subcaptionbox*{III-iii}[0.3\linewidth]
{
\begin{tikzpicture}[scale=0.6]
\coordinate (u1) at (5,1);
\coordinate (u2) at (1,2);
\coordinate (u3) at (-2,4);

\foreach \i in {1,...,3}
{
\path let \p1=(u\i), \n1={veclen(\x1,\y1)} in
	coordinate (v\i) at (${1/\n1}*(u\i)$);
}

\foreach \v[remember=\v as \lastv (initially v1)] in {v1,v2,v3}
	\path[top color=gray!25!, bottom color=gray!60!] 
		(0,0)--($3cm*(\lastv)$)--($3cm*(\v)$)--cycle;

\foreach \v/\dmax in {v1/3,v2/1,v3/2}
{
	\draw[thick,color=black] (0,0) -- ($3cm*(\v)$);
	\foreach \d in {1,...,\dmax}
		\fill (${1cm*((1+(\d-1)*0.5)}*(\v)$) circle (0.5ex);
}

\end{tikzpicture}   
}
\end{figure}

\noindent
\emph{Case III-i}.
Here we have $\lambda = \cone(w_1, w_4)$.
Let $v \in \ZZ^2$ be a primitive vector on~$\sigma_2$.
\Cref{prop:hypmovforCY} and \cref{rem:weaklyCY}~(i)
tell us $\mu \in \lambda^\circ \cup \sigma_2$.
This allows us to apply \cref{lem:threegenerateCY} to $w_i, w_4, w_5$
for $i = 1,2,3$.
From this we infer $\det(u_i, v) = 1$ for $i = 1,2,3$.
In particular $u_1, u_2, u_3$ are primitive,
thus $u_1 = u_2 = u_3$.
Applying \cref{lem:threegenerateCY} to the triple~$w_1, w_2, w_6$
shows $\det(u_1, u_6) = 1$.
A suitable admissible coordinate change amounts to
$v = (0,1)$ and
\[
	\free{Q} = \begin{bmatrix*}[r]
	1 & 1 & 1 &   0 &   0 & -a_6 \\
	0 & 0 & 0 & b_4 & b_5 &    1
	\end{bmatrix*}, \qquad
	a_6, b_4, b_5 \in \ZZ_{\geq 1}
\]
We may assume $b_4 \leq b_5$.
To proceed we have to take the position of $\mu$ into account.

\medskip
\noindent
Assume $\mu \in \lambda^\circ$.
Then we may apply \cref{lem:threegenerateCY,lem:2on1ray} to
the two triples~$w_1, w_2, w_4$ and~$w_1, w_2, w_5$.
We obtain that $u_4$ and $u_5$
both are primitive, hence 
\[
	u_4 = u_5 = v = (0, 1).
\]
From \cref{rem:weaklyCY}~(i) we infer
$\alpha = (3-a_6, 3)$.
Since $\mu$ lives in the relative interior
of~$\lambda$, which is the positive orthant,
we end up with $a_6 = 1, 2$.
We show that $K$ is torsion free in both cases.

\begin{itemize}
\item $a_6 = 1$.
The free parts of the specifying data are given as
\[
	\free{Q} = \begin{bmatrix*}[r]
	1 & 1 & 1 & 0 & 0 & -1 \\
	0 & 0 & 0 & 1 & 1 &  1
	\end{bmatrix*}, \qquad
	\alpha = (2, 3).
\]
\Cref{lem:twovarmonomial} ensures that $T_1^2 T_4^3$ is of degree $\mu$.
Moreover \cref{lem:threegenerateCY}
shows that both triples $w_1, w_4, w_5$ and $w_1, w_2, w_4$ generate $K$ as a group.
Applying \cref{lem:rho2torsbound}
to $w_1, w_4, w_5$ and $T_1^2 T_4^3$ yields
$t \mid 3$. Again \cref{lem:rho2torsbound},
this time applied to $w_1, w_2, w_4$ and $T_1^2 T_3^3$
shows $t \mid 2$. Altogether $t = 1$, thus $K$ is torsion-free.

\item $a_6 = 2$.
The free parts of the specifying data are given as
\[
	\free{Q} = \begin{bmatrix*}[r]
	1 & 1 & 1 & 0 & 0 & -2 \\
	0 & 0 & 0 & 1 & 1 &  1
	\end{bmatrix*}, \qquad
	\alpha = (1, 3).
\]
\Cref{lem:twovarmonomial} ensures that $T_1^1 T_4^3$ of degree $\mu$.
Moreover \cref{lem:threegenerateCY}
shows that $w_1, w_2, w_4$ generate $K$ as a group.
Applying \cref{lem:rho2torsbound}
to $w_1, w_2, w_4$ and $T_1 T_4^3$
shows $t = 1$ i.e.\ $K$ is torsion-free.
\end{itemize}
Eventually this leads to specyifing data as in 
Numbers~\ref{CY3fold:III-i-1} and~\ref{CY3fold:III-i-2}
from \cref{thm:CY3folds}.

\medskip
\noindent
Assume $\mu \in \sigma_2$.
Recall that $v = (0,1)$ spans the ray $\sigma_2$.
So here we have $\alpha_1 = 0$.
From \cref{rem:weaklyCY}~(i)
we obtain $a_6 = 3$ and $\alpha_2 = b_4 + b_5 + 1$.
\Cref{lem:raypowerCY} yields $b_4, b_5 \mid \alpha_2$.
Applying \cref{lem:threegenerateCY} to $w_1, w_4, w_5$ shows
$\gcd(b_4, b_5) = 1$.
We conclude
\[
	b_4 b_5 \mid \alpha_2 = b_4 + b_5 + 1.
\]
This implies $b_5 \mid b_4 + 1$. Moreover we 
deduce $b_4 \leq 3$.
We discuss the resulting cases:
\begin{itemize}[itemsep=1ex]
\item
$b_4 = 1$:
From $b_5 \mid b_4 + 1 = 2$ we deduce $b_5 = 1,2$.
For $b_5 = 1$ we have
\[
	Q^0 = \begin{bmatrix*}[r]
	1 & 1 & 1 & 0 & 0 & -3 \\
	0 & 0 & 0 & 1 & 1 &  1
	\end{bmatrix*}, \quad
	\alpha = (0, 3).
\]
Suppose that $K$ is torsion-free.
Then $w_4 = w_5$ holds.
Reversing the order of the variables by applying a suitable
admissible coordinate change enables us to use \cref{lem:w4onrho2CY}.
This forces two of the rays $\sigma_i$ to coincide; a contradiction.
So $K$ has torsion.
From $\alpha= 3 u_4$ and \cref{lem:torsboundpower} we obtain $t \mid 3$,
hence $t = 3$. So we have $K = \ZZ^2 \times \ZZ / 3 \ZZ$. 
Using $\det(u_1, u_6) = 1$ enables us to apply
a suitable admissible coordinate change
such that $\zeta_1 = \zeta_6 = 0$.
Now \cref{lem:threegenerateCY} shows that 
both triples $w_1, w_2, w_6$ and $w_1, w_3, w_6$
generate $K$ as a group.
From this we infer that $\zeta_2$ and $\zeta_3$
both are generators for $\ZZ \times \ZZ / 3 \ZZ$.
\Cref{lem:threegenerateCY} yields that $w_2, w_3, w_6$
form a generating set for $K$ as well.
This forces $\zeta_2 \neq \zeta_3$. 
Otherwise $w_2 = w_3$ holds, thus $K$ is spanned by two elements;
a contradiction.
Similarly, \cref{lem:threegenerateCY} applied to $w_1, w_4, w_6$
and $w_1, w_5, w_6$ yields that $\zeta_4$ and $\zeta_5$
both are generators for $\ZZ / 3 \ZZ$. 
Moreover applying \cref{lem:threegenerateCY} to $w_1, w_4, w_5$
ensures $\zeta_4 \neq \zeta_5$.
After suitably reordering $T_2$, $T_3$ and $T_4, T_5$
we arrive at Number~\ref{CY3fold:III-i-3-tors} from \cref{thm:CY3folds}.
	
We turn to $b_5 = 2$.
Here the free parts of degree matrix and relation degree
are given by
\[
	Q^0 = \begin{bmatrix*}[r]
	1 & 1 & 1 & 0 & 0 & -3 \\
	0 & 0 & 0 & 1 & 2 &  1
	\end{bmatrix*}, \quad
	\alpha = (0, 4).
\]
Note $\alpha = 2 u_5$. From \cref{lem:torsboundpower} we infer $t \mid 2$,
hence $K$ is torsion-free according to \cref{lem:tors24}.
Moreover every $\mu$-homogeneous polynomial not depending on $T_6$
is a linear combination over the monomials $T_4^4$, $T_4^2 T_5$, $T_5^2$,
thus reducible.
This implies that $T_6 \in R$ is not prime.
A contradiction.

\item
$b_4 = 2$:
From $b_5 \mid b_4 + 1 = 3$ and $b_4 \leq b_5$ follows $b_5 = 3$.
This leads to
\[
	Q^0 = \begin{bmatrix*}[r]
	1 & 1 & 1 & 0 & 0 & -3 \\
	0 & 0 & 0 & 2 & 3 &  1
	\end{bmatrix*}, \quad
	\alpha = (0, 6).
\]
Observe $\alpha = 3 u_2 = 2 u_3$.
\Cref{lem:torsboundpower} yields $t \mid 2$ and $t \mid 3$,
hence $t = 1$. So~$K$ is torsion-free.
We end up with Number~\ref{CY3fold:III-i-5} from \cref{thm:CY3folds}.

\item
$b_4 = 3$:
From $b_5 \mid b_4 + 1 = 4$ and $b_4 \leq b_5$ we infer $b_5 = 4$.
This implies $\alpha_2 = 8$; a contradiction to $b_4 \mid \alpha_2$.
\end{itemize}

\bigskip
\noindent
\emph{Case III-ii}. 
Here, we have $\lambda = \cone(w_2, w_5)$.
\Cref{prop:hypmovforCY} says $\mu \in \sigma_2$.
Let $v \in \ZZ^2$ be
a primitive vector on $\sigma_2$.
Applying \cref{lem:threegenerateCY} to $w_2, w_3, w_5$
as well as $w_2, w_3, w_6$ shows
$\det(v, u_5) = 1$ and $\det(v, u_6) = 1$.
In particular $u_5, u_6$ are primitive and lie on the same ray hence coincide.
Again by \cref{lem:threegenerateCY}, now applied 
to $w_1, w_5, w_6$, we obtain $\det(u_1, u_5) = 1$.
A suitable admissible coordinate change amounts to
$v = (1,0)$ and~$u_5 = (0,1)$.
As a result the free part $Q^0$ of the degree matrix $Q$ is of the form
\[
	Q^0 = \begin{bmatrix*}[r]
		 1 & a_2 & a_3 & a_4 & 0 & 0 \\
		-2 &   0 &   0 &   0 & 1 & 1
	\end{bmatrix*}, \quad
	a_2, a_3, a_4 \in \ZZ_{\geq 1}.
\]
Note that the second coordinate of $u_1$ is determined by
$\alpha_2 = 0$ and \cref{rem:weaklyCY}~(i).
Furthermore, we may assume $a_2 \leq a_3 \leq a_4$.
\Cref{lem:raypowerCY} shows that $\alpha_1$ is divisible by each
of $a_2, a_3, a_4$.
From applying \cref{lem:threegenerateCY} to all triples $w_i, w_j, w_6$
where $2 \leq i < j \leq 4$ we infer that $a_2, a_3, a_4$
are pairwise coprime.
This leads to
\[
	a_2 a_3 a_4 \mid a_2 + a_3 + a_4 + 1.
\]
According to \cref{rem:sumprodforCY} we have $a_2 = 1$ and
one of the following two configurations
\[
	a_3 = 1, \quad a_3 = 2 \text{ and } a_4 = 3.
\]
Note that $a_3 = 2$ and $a_4 = 3$
amounts to $\alpha_1 = 7$;
a contradiction to $a_3 \mid \alpha_1$.
So we have $a_3 = 1$.
Then $a_4 \mid \alpha_1 = 3 + a_4$ holds.
We conclude $a_4 \mid 3$ i.e.\ $a_4 = 1, 3$.
We show that $K$ is torsion-free in both cases:
\begin{itemize}
\item $a_4 = 1$. Here we have
\[
	Q^0 = \begin{bmatrix*}[r]
		 1 & 1 & 1 & 1 & 0 & 0 \\
		-2 & 0 & 0 & 0 & 1 & 1
	\end{bmatrix*}, \quad
	\alpha = (4, 0).
\]
Note $\alpha = 4 u_2$, thus $t \mid 4$ by \cref{lem:torsboundpower}.
Now \cref{lem:tors24} ensures that $K$ is torsion-free.

\item $a_4 = 2$. Here we have
\[
	Q^0 = \begin{bmatrix*}[r]
		 1 & 1 & 1 & 3 & 0 & 0 \\
		-2 & 0 & 0 & 0 & 1 & 1
	\end{bmatrix*}, \quad
	\alpha = (6, 0).
\]
Note $\alpha = 2 u_4$, thus $t \mid 2$ by \cref{lem:torsboundpower}.
Now \cref{lem:tors24} ensures that $K$ is torsion-free.
\end{itemize}

We have arrived at
Numbers~\ref{CY3fold:III-ii-1} and~\ref{CY3fold:III-ii-2}
from \cref{thm:CY3folds}.

\bigskip
\noindent
\emph{Case III-iii}.
From \cref{lem:combminrayCY}~(i) we obtain
\[
	u_1 = u_2 = u_3, \qquad
	u_5 = u_6, \qquad
	\det(u_1, u_6) = 1.
\]
A suitable admissible coordinate change brings the degree matrix
into the following form
\[
	Q^0 = \begin{bmatrix}
	1 & 1 & 1 & a_4 & 0 & 0 \\
	0 & 0 & 0 & b_4 & 1 & 1
	\end{bmatrix}, \qquad
	a_4, b_4 \in \ZZ_{\geq 1}.
\]
Moreover, \cref{prop:hypmovforCY} tells us
$\mu \in \cone(w_1, w_4)^\circ$ or $\mu \in \varrho_4$.
Let us first assume $\mu \in \cone(w_1, w_4)^\circ$.
According to \cref{cor:gitconeCY} we have GIT-cones
\[
    \eta_1 = \cone(w_1, w_4), \qquad
    \eta_2 = \cone(w_4, w_5),
\]
both of them giving rise to a smooth variety $X(\eta_i)$;
see \cref{CYSQMsmooth}.
We obtain that $K$ is torsion-free
by applying \cref{lem:torsfree} to $X(\eta_2)$.
Applying \cref{lem:combminrayCY}~(ii) gives $u_4 = u_1 + u_6 = (1,1)$.
We have arrived at Numbers~\ref{CY3fold:III-iii-1a}
and~\ref{CY3fold:III-iii-1b} from \cref{thm:CY3folds}.

The next step is to consider $\mu \in \varrho_4$.
\Cref{lem:raypowerCY} provides us with some $k \in \ZZ_{\geq 2}$
such that $\mu = k w_4$ holds.
Using \cref{rem:weaklyCY}~(i) gives $k b_4 = \alpha_2 = b_4 + 2$.
We conclude $b_4 \mid 2$.
This leads to one of the following two configurations
\[
	k = 3 \text{ and } b_4 = 1, \qquad
	k = 2 \text{ and } b_4 = 2.
\]
Suppose $k = 3$.
Using \cref{rem:weaklyCY}~(i) again shows
$3 a_4 = 3 + a_4$, equivalently $2 a_4 = 3$. A contradiction.
We must have $k = 2$ and $b_4 = 2$.
Here \cref{rem:weaklyCY}~(i) implies $2 a_4 = 3 + a_4$,
thus $a_4 = 3$.
We have
\[
	Q^0 = \begin{bmatrix}
	1 & 1 & 1 & 3 & 0 & 0 \\
	0 & 0 & 0 & 2 & 1 & 1
	\end{bmatrix}, \qquad
	\alpha = (6, 4).
\]
From $\alpha = 2 u_4$ we infer $t \mid 2$ by \cref{lem:torsboundpower}.
Thus \cref{lem:tors24} yields that $K$ is torsion-free.
This amounts to Number~\ref{CY3fold:III-iii-2}
from \cref{thm:CY3folds}.

\subsection*{Part IV}
This parts deals with the case of
$w_1, \dotsc, w_6$ being disposed
on four rays according to
the integer partition $1 + 1 + 2 + 2 = 6$.
A suitable admissible coordinate change
leads to one of the subsequent constellations:
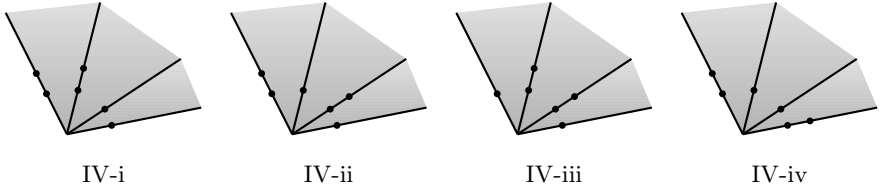
\begin{figure}[H]
\centering
\renewcommand\thesubfigure{\roman{subfigure}}
\subcaptionbox*{IV-i}[0.225\linewidth]
{
\begin{tikzpicture}[scale=0.6]
\coordinate (u1) at (5,1);
\coordinate (u2) at (3,2);
\coordinate (u3) at (1,4);
\coordinate (u4) at (-2,4);

\foreach \i in {1,...,4}
{
\path let \p1=(u\i), \n1={veclen(\x1,\y1)} in
	coordinate (v\i) at (${1/\n1}*(u\i)$);
}

\foreach \v[remember=\v as \lastv (initially v1)] in {v1,v2,v3,v4}
	\path[top color=gray!25!, bottom color=gray!60!] 
		(0,0)--($3cm*(\lastv)$)--($3cm*(\v)$)--cycle;

\foreach \v/\dmax in {v1/1,v2/1,v3/2,v4/2}
{
	\draw[thick,color=black] (0,0) -- ($3cm*(\v)$);
	\foreach \d in {1,...,\dmax}
		\fill (${1cm*((1+(\d-1)*0.5)}*(\v)$) circle (0.5ex);
}

\end{tikzpicture}   
} %
\subcaptionbox*{IV-ii}[0.225\linewidth]
{
\begin{tikzpicture}[scale=0.6]
\coordinate (u1) at (5,1);
\coordinate (u2) at (3,2);
\coordinate (u3) at (1,4);
\coordinate (u4) at (-2,4);

\foreach \i in {1,...,4}
{
\path let \p1=(u\i), \n1={veclen(\x1,\y1)} in
	coordinate (v\i) at (${1/\n1}*(u\i)$);
}

\foreach \v[remember=\v as \lastv (initially v1)] in {v1,v2,v3,v4}
	\path[top color=gray!25!, bottom color=gray!60!] 
		(0,0)--($3cm*(\lastv)$)--($3cm*(\v)$)--cycle;

\foreach \v/\dmax in {v1/1,v2/2,v3/1,v4/2}
{
	\draw[thick,color=black] (0,0) -- ($3cm*(\v)$);
	\foreach \d in {1,...,\dmax}
		\fill (${1cm*((1+(\d-1)*0.5)}*(\v)$) circle (0.5ex);
}

\end{tikzpicture}   
} %
\subcaptionbox*{IV-iii}[0.225\linewidth]
{
\begin{tikzpicture}[scale=0.6]
\coordinate (u1) at (5,1);
\coordinate (u2) at (3,2);
\coordinate (u3) at (1,4);
\coordinate (u4) at (-2,4);

\foreach \i in {1,...,4}
{
\path let \p1=(u\i), \n1={veclen(\x1,\y1)} in
	coordinate (v\i) at (${1/\n1}*(u\i)$);
}

\foreach \v[remember=\v as \lastv (initially v1)] in {v1,v2,v3,v4}
	\path[top color=gray!25!, bottom color=gray!60!] 
		(0,0)--($3cm*(\lastv)$)--($3cm*(\v)$)--cycle;

\foreach \v/\dmax in {v1/1,v2/2,v3/2,v4/1}
{
	\draw[thick,color=black] (0,0) -- ($3cm*(\v)$);
	\foreach \d in {1,...,\dmax}
		\fill (${1cm*((1+(\d-1)*0.5)}*(\v)$) circle (0.5ex);
}

\end{tikzpicture}   
} %
\subcaptionbox*{IV-iv}[0.225\linewidth]
{
\begin{tikzpicture}[scale=0.6]
\coordinate (u1) at (5,1);
\coordinate (u2) at (3,2);
\coordinate (u3) at (1,4);
\coordinate (u4) at (-2,4);

\foreach \i in {1,...,4}
{
\path let \p1=(u\i), \n1={veclen(\x1,\y1)} in
	coordinate (v\i) at (${1/\n1}*(u\i)$);
}

\foreach \v[remember=\v as \lastv (initially v1)] in {v1,v2,v3,v4}
	\path[top color=gray!25!, bottom color=gray!60!] 
		(0,0)--($3cm*(\lastv)$)--($3cm*(\v)$)--cycle;

\foreach \v/\dmax in {v1/2,v2/1,v3/1,v4/2}
{
	\draw[thick,color=black] (0,0) -- ($3cm*(\v)$);
	\foreach \d in {1,...,\dmax}
		\fill (${1cm*((1+(\d-1)*0.5)}*(\v)$) circle (0.5ex);
}

\end{tikzpicture}   
} %
\end{figure}

\noindent
\emph{Case IV-i}.
Here, \cref{prop:hypmovforCY} tells us $\mu \in \sigma_3$.
As a result, \cref{cor:gitconeCY} provides us with two GIT-cones
\[
	\eta_1 = \cone(w_2, w_3), \qquad
	\eta_2 = \cone(w_3, w_5).
\]
\Cref{CYSQMsmooth} ensures that the
associated varieties $X(\eta_1)$ and $X(\eta_2)$
both are smooth.
Let $v \in \ZZ^2$ denote the primitive
lattice vector lying on $\sigma_3$.
Consider $X(\eta_2)$. Applying \cref{lem:threegenerateCY,lem:2on1ray}
to the triples $w_3, w_4, w_5$ and $w_3, w_4, w_6$ yields
$u_5 = u_6$ and $\det(v, u_5) = 1$.
Thus we may apply a suitable admissible coordinate change such that
$v = (1,0)$ and $u_5 = u_6 = (0,1)$.
We apply \cref{lem:threegenerateCY} again, this time
to $w_1, w_5, w_6$ and $w_2, w_5, w_6$.
This shows that the first coordinate of both $u_1$ and $u_2$
equals one.
Now, consider $X(\eta_1)$.
We apply \cref{lem:threegenerateCY} to $w_1, w_3, w_4$,
hence obtain $u_1 = (1, -1)$.
Analogously, we obtain $u_2 = (1, -1)$, thus $u_1 = u_2$.
This contradicts $\sigma_1 \neq \sigma_2$.

\bigskip
\noindent
\emph{Case IV-ii}.
\Cref{prop:hypmovforCY} says $\mu \in \cone(w_2, w_4)$.
First, we assume $\mu \in \varrho_4 = \sigma_3$.
Then \cref{cor:gitconeCY} ensures $\lambda = \cone(w_2, w_5)$.
Let $v \in \ZZ^2$ be the primitive lattice vector on~$\sigma_2$.
Applying \cref{lem:threegenerateCY,lem:2on1ray} to
all four triples
\[
(w_2, w_3, w_5), \quad (w_2, w_3, w_6), \quad
(w_2, w_5, w_6), \quad (w_3, w_5, w_6)
\]
shows that $u_2, u_3, u_5, u_6$ are primitive,
thus $u_2 = u_3$ and $u_5 = u_6$.
Additionally we obtain $\det(u_2, u_5) = 1$.
\Cref{lem:threegenerateCY} again, 
this time applied to $w_1, w_5, w_6$,
tells us $\det(u_1, u_5) = 1$.
A suitable admissible coordinate change eventually amounts to
\[
	Q^0 = \begin{bmatrix*}[r]
	   1 & 1 & 1 & a_4 & 0 & 0 \\
	-b_1 & 0 & 0 & b_4 & 1 & 1
	\end{bmatrix*},
	\qquad
	a_4, b_1, b_4 \in \ZZ_{\geq 1}.
\]
From \cref{rem:weaklyCY}~(i) we infer
$\alpha = (a_4 + 3,\, b_4 - b_1 + 2)$.
\Cref{lem:raypowerCY} provides us with some $k \in \ZZ_{\geq 2}$ such that 
$\mu = k w_4$. In particular $a_4 \mid \alpha_1 = a_4 + 3$.
This implies $a_4 = 1, 3$.
Suppose $a_4 = 1$.
Then $k = 4$ holds.
This leads to
$4b_4 = \alpha_2 = b_4 - b_1 + 2$, hence
$3 b_4 = 2 - b_1$.
A contradiction to $b_1, b_4 \geq 1$.
We are left with $a_4 = 3$ and $k = 2$.
Inserting into $\alpha = k u_3$ gives
$2b_4 = b_4 - b_1 + 2$, thus
$b_4 = 2 - b_1$.
This forces $b_1 = 1$ and $b_4 = 1$
due to $b_1, b_4 \geq 1$.
Moreover $k = 2$ implies $t \mid 2$ by \cref{lem:torsboundpower}.
Thus $K$ is torsion-free according to \cref{lem:tors24}.
We have arrived at Number~\ref{CY3fold:IV-ii-1}
from \cref{thm:CY3folds}.

We turn to the case $\mu \notin \sigma_3$.
Here \cref{cor:gitconeCY} provides us with two GIT-cones
\[
	\eta_1 = \cone(w_2, w_4), \qquad
	\eta_2 = \cone(w_4, w_5).
\]
According to \cref{CYSQMsmooth} the according varieties $X(\eta_1)$
and $X(\eta_2)$ both are smooth.
Consider $X(\eta_2)$.
\Cref{lem:twofacesCY} applied to $w_4, w_5$ and $w_4, w_6$
yields $\det(u_4, u_5) = 1$ as well as $\det(u_4, u_6) = 1$.
Besides, \cref{lem:threegenerateCY,lem:2on1ray} applied to $w_1, w_5, w_6$
give us $\det(u_1, u_5) = 1$.
Now consider $X(\eta_1)$.
Let $v \in \ZZ^2$ be the primitive vector contained in~$\sigma_2$.
Applying \cref{lem:threegenerateCY,lem:2on1ray} to
$w_2, w_3, w_4$ and $w_2, w_3, w_5$ shows $\det(v, u_4) = 1$
and $\det(v, u_5) = 1$.
Now we apply an admissible coordinate change such that $v = (1,0)$
and $u_5 = (0,1)$ holds.
Taking the determintal equations from above into account
amounts to the following degree matrix
\[
	Q^0 = \begin{bmatrix*}[r]
	   1 & a_2 & a_3 & 1 & 0 & 0 \\
	-b_1 &   0 &   0 & 1 & 1 & 1
	\end{bmatrix*},
	\qquad
	a_2, a_3, b_1 \in \ZZ_{\geq 1}.
\]
We may assume $a_2 \leq a_3$.
From \cref{rem:weaklyCY}~(i) follows $\alpha_2 = 3 - b_1$.
\Cref{prop:hypmovforCY} guarantees that $\mu$ lives in the positive orthant,
hence $b_1 \leq 3$.
Furthermore, \cref{lem:twofacesCY} applied w.r.t\ $X(\eta_1)$ and
the pairs $w_2, w_5$ and $w_3, w_5$
shows $a_2, a_3 \mid \alpha_1$.
Applying \cref{lem:threegenerateCY} to $w_2, w_3, w_5$
shows $\gcd(a_2, a_3) = 1$.
Consequently $a_2 a_3 \mid \alpha_1 = a_2 + a_3 + 2$ holds.
We end up with $a_2 = 1$ and $a_3 = 1, 3$.

Let us discuss the case $a_3 = 1$.
Here specifying data looks as follows
\[
	Q^0 = \begin{bmatrix*}[r]
	   1 & 1 & 1 & 1 & 0 & 0 \\
	-b_1 & 0 & 0 & 1 & 1 & 1
	\end{bmatrix*},
	\quad
	\alpha = (4,\, 3 - b_1), \qquad
	b_1 \in \{1, 2, 3\}.
\]
Suppose $b_1 = 3$.
This implies $\alpha = (4, 0) = 4 u_2$.
\Cref{lem:torsboundpower,lem:tors24} yield
that~$K$ is torsion-free.
So $w_2 = w_3$ holds.
Note that $\alpha_2 = 0$ means $\mu \in \varrho_2$.
In this situation \cref{lem:w4onrho2CY} says $w_4 \in \varrho_2$.
A contradiction to $\sigma_2 \neq \sigma_3$.
So we have $b_1 = 1, 2$.
Observe $\mu \in \eta_1^\circ$.
Appyling \cref{lem:torsfree} to $X(\eta_2)$
guarantees that $K$ is torsion-free.
We end up with Numbers~\ref{CY3fold:IV-ii-2a}
to~\ref{CY3fold:IV-ii-3b} from \cref{thm:CY3folds}.

We turn to $a_3 = 3$.
Here we have $\alpha_1 = 6$.
According to \cref{lem:twofacesCY} applied to $X(\eta_1)$ and $w_3, w_4$,
there must be some monomial
$T_3^{l_3} T_4^{l_4}$ of degree $\mu$ because of $\det(u_3, u_4) = 3$.
As the second coordinate of $u_3$ vanishes,
$l_4 = \alpha_2 = 3 - b_1$ holds.
Inserting this into the equation $\alpha_1 = l_3 a_3 + l_4 a_4$
yields $3l_3 + 3 - b_1 = \alpha_1 = 6$.
This forces $b_1$ to be divisible by~$3$, hence $b_1 = 3$.
We arrive at the following data
\[
	Q^0 = \begin{bmatrix*}[r]
	 1 &  1 &  3 & 1 & 0 & 0 \\
	-3 &  0 &  0 & 1 & 1 & 1
	\end{bmatrix*}, \qquad
	\alpha = (6, 0).
\]
Observe that this grading does not admit any
monomial of the form $T_1^{l_1} T_4^{l_4}$
of degree $\mu$. 
Thus $\det(u_1, u_4) = 1$ by \cref{lem:twofacesCY}
applied to $X(\eta_1)$ and $w_1, w_4$.
A contradiction.

\bigskip
\noindent
\emph{Case IV-iii}.
\Cref{cor:gitconeCY} ensures $\lambda = \cone(w_3, w_5)$.
Let $v, v' \in \ZZ^2$ be the primitive ray generators
of $\sigma_2$, $\sigma_3$.
We may apply \cref{lem:threegenerateCY,lem:2on1ray} to at least one
of the triples $w_2, w_3, w_4$ and $w_2, w_4, w_5$.
From this we infer $\det(v, v') = 1$.
A suitable admissible coordinate change leads to
$v = (1,0)$ and $v' = (0,1)$.
Applying \cref{lem:threegenerateCY} to $w_2, w_3, w_6$
yields $w_6 = (-a_6, 1)$ for some $a_6 \in \ZZ_{\geq 1}$.
Similarly, one obtains $w_1 = (1, -b_1)$ with $b_1 \in \ZZ_{\geq 1}$.
Counter-clockwise orientation yields $\det(w_1, w_6) = 1 - a_6 b_1 \leq 0$.
We conclude $b_1 = a_6 = 1$, hence $w_1 = - w_6$.
This contradicts $\Eff(R)$ being pointed.

\bigskip
\noindent
\emph{Case IV-iv}.
We have $\mu \in \cone(w_3, w_4)$ by \cref{prop:hypmovforCY}.
Suppose $\mu \in \cone(w_3, w_4)^\circ$.
\Cref{CYSQMsmooth} allows us to apply \cref{lem:combminrayCY}~(ii).
From this we infer $u_3 = u_4$,
thus $\sigma_2 = \sigma_3$; a contradiction.
So we have $\mu \in \sigma_2 \cup \sigma_3$.
Taking the symmetry in the geometric constellation
of $w_1, \dotsc, w_6$ into account
a suitable admissible coordinate change
amounts to $\mu \in \sigma_2$.
\Cref{lem:combminrayCY} yields $u_1 = u_2$, $u_5 = u_6$
and $u_4 = u_1 + u_6$. Furthermore we obtain $\det(u_1, u_6) = 1$.
After applying another suitable admissible coordinate change 
the degree matrix is of the following form
\[
	Q^0 = \begin{bmatrix*}[r]
	1 & 1 & a_3 & 1 & 0 & 0 \\
	0 & 0 & b_3 & 1 & 1 & 1
	\end{bmatrix*}, \quad
	a_3, b_3 \in \ZZ_{\geq 1}.
\]
From $X$ being Calabi-Yau we infer $\alpha = (a_3 + 3,\, b_3 + 3)$;
see \cref{rem:weaklyCY}.
Besides \cref{lem:raypowerCY} provides us with some $k \in \ZZ_{\geq 2}$
such that $\mu = k w_3$ holds.
Altogether we obtain $(k-1) a_3 = 3 = (k-1) b_3$, hence $a_3 = b_3$.
This contradicts $\sigma_2 \neq \sigma_3$.

\subsection*{Part VII}
We work out the the constellation where the Cox ring
generator degrees $w_1, \dotsc, w_6$ lie on pairwise
different rays i.e.\ we have $\sigma_i = \varrho_i$ for all $i = 1, \dotsc, 6$.
\Cref{prop:hypmovforCY} says $\mu \in \cone(w_3, w_4)$.
After a applying a suitable admissible coordinate change
we have either $\mu \in \cone(w_3, w_4)^\circ$ 
or $\mu \in \varrho_3$.
\begin{figure}[H]
\centering
\renewcommand\thesubfigure{\roman{subfigure}}
\subcaptionbox*{VII-a: $\mu \in \cone(w_3, w_4)^\circ$}[0.3\linewidth]
{
\begin{tikzpicture}[scale=0.6]
\coordinate (u1) at (5,1);
\coordinate (u2) at (2,1);
\coordinate (u3) at (1,1);
\coordinate (u4) at (1,3);
\coordinate (u5) at (-1,10);
\coordinate (u6) at (-1,2);

\foreach \i in {1,...,6}
{
\path let \p1=(u\i), \n1={veclen(\x1,\y1)} in
	coordinate (v\i) at (${1/\n1}*(u\i)$);
}

\foreach \v[remember=\v as \lastv (initially v1)] in {v1,v2,v3,v4,v5,v6}
	\path[top color=gray!25!, bottom color=gray!60!] 
		(0,0)--($3cm*(\lastv)$)--($3cm*(\v)$)--cycle;

\foreach \v/\dmax in {v1/1,v2/1,v3/1,v4/1,v5/1,v6/1}
{
	\draw[thick,color=black] (0,0) -- ($3cm*(\v)$);
	\foreach \d in {1,...,\dmax}
		\fill (${1cm*((1+(\d-1)*0.5)}*(\v)$) circle (0.5ex);
}

\filldraw[fill=white, draw=black] ($2.5cm*0.5*(v3) + 2.5cm*0.5*(v4)$) circle (0.5ex);
\end{tikzpicture}   
} %
\subcaptionbox*{VII-b: $\mu \in \varrho_3$}[0.3\linewidth]
{
\begin{tikzpicture}[scale=0.6]
\coordinate (u1) at (5,1);
\coordinate (u2) at (2,1);
\coordinate (u3) at (1,1);
\coordinate (u4) at (1,3);
\coordinate (u5) at (-1,10);
\coordinate (u6) at (-1,2);

\foreach \i in {1,...,6}
{
\path let \p1=(u\i), \n1={veclen(\x1,\y1)} in
	coordinate (v\i) at (${1/\n1}*(u\i)$);
}

\foreach \v[remember=\v as \lastv (initially v1)] in {v1,v2,v3,v4,v5,v6}
	\path[top color=gray!25!, bottom color=gray!60!] 
		(0,0)--($3cm*(\lastv)$)--($3cm*(\v)$)--cycle;

\foreach \v/\dmax in {v1/1,v2/1,v3/1,v4/1,v5/1,v6/1}
{
	\draw[thick,color=black] (0,0) -- ($3cm*(\v)$);
	\foreach \d in {1,...,\dmax}
		\fill (${1cm*((1+(\d-1)*0.5)}*(\v)$) circle (0.5ex);
}

\filldraw[fill=white, draw=black] ($2.5cm*0.5*(v3) + 2.5cm*0.5*(v3)$) circle (0.5ex);
\end{tikzpicture}   
} %
\end{figure}

\noindent
\emph{Case VII-a}.
Here, we assume $\mu \in \cone(w_3, w_4)^\circ$.
According to \cref{cor:gitconeCY} 
the cones
\[
	\eta_1 = \cone(w_2, w_3), \qquad
	\eta_2 = \cone(w_3, w_4), \qquad
	\eta_3 = \cone(w_4, w_5)
\]
are GIT-cones leading to smooth varieties $X(\eta_i)$;
see also \cref{CYSQMsmooth}.
Let us consider $X(\eta_1)$.
\Cref{lem:twofacesCY} applied to $w_1, w_3$ and $w_2, w_3$
yields $\det(u_1, u_3) = 1$ and $\det(u_2, u_3) = 1$.
Thus a suitable admissible coordinate change leads to
\[
	Q^0 = \begin{bmatrix*}[r]
		   1 & 1 & 0 & - a_4 & -a_5 & -a_6 \\
		-b_1 & 0 & 1 &   b_4 &  b_5 &  b_6
	\end{bmatrix*}, \quad
	a_i, b_i \in \ZZ_{\geq 1}.
\]
Consider $X(\eta_3)$.
Applying \cref{lem:twofacesCY} to $w_4, w_5$ and $w_4, w_6$
gives $\det(u_4, u_5) = 1$ and $\det(u_4, u_6) = 1$.
Since $a_4 \neq 0$, this is equivalent to
\begin{equation}
\label{eq:CY:VII-a:w5w6}
	b_5 = \dfrac{a_5 b_4 - 1}{a_4}, \qquad
	b_6 = \dfrac{a_6 b_4 - 1}{a_4}.
\end{equation}
Now consider $X(\eta_2)$.
Applying \cref{lem:twofacesCY} to the pair $w_3, w_i$
for $i = 4, 5, 6$ shows that $\alpha_1$ is divisible
by each of $a_4, a_5, a_6$.
Moreover, \cref{lem:threegenerateCY} applied to
$w_2, w_i, w_j$ where $4 \leq i < j \leq 6$
ensures that $a_4, a_5, a_6$ are pairwise coprime.
Together with \cref{rem:weaklyCY}~(i) we obtain
\[
	a_4 a_5 a_6 \mid \alpha_1 = a_4 + a_5 + a_6 - 2.
\]
One quickly checks that this forces two of $a_4, a_5, a_6$
to equal one.
Suppose $a_5 = a_6 = 1$.
Then \cref{eq:CY:VII-a:w5w6} implies $b_5 = b_6$,
thus $u_5 = u_6$. A contradiction.
So we must have $a_4 = 1$, in particular
\begin{equation}
\label{eq:CY:VII-a:w5w6-2}
	b_5 = a_5b_4 - 1, \qquad b_6 = a_6 b_4 - 1.
\end{equation}
Furthermore, \cref{lem:twofacesCY} applied
to $w_2, w_j$ gives $b_j \mid \alpha_2$ for $j = 4, 5, 6$.
In addition, applying \cref{lem:threegenerateCY} 
to all triples $w_2, w_i, w_j$ where $4 \leq i < j \leq 6$
shows that $b_4, b_5, b_6$ are pairwise coprime.
Once again by \cref{rem:weaklyCY}~(i) we obtain
\begin{equation}
\label{eq:CY:VII-a:CYmu2}
	b_4 b_5 b_6 \mid \alpha_2 = b_4 + b_5 + b_6 + 1 - b_1.
\end{equation}
Note that the right-hand side is positive due to the position of $\mu$.
From this we deduce $b_4 b_5 b_6 \leq b_4 + b_5 + b_6$.
According to \cref{rem:sumprodforCY} this inequation implies
that either two of $b_4, b_5, b_6$ equal one
or $\{b_4, b_5, b_6\} = \{1,2,3\}$.

We exclude the first option. 
Here we have $b_5 \neq b_6$ by \cref{eq:CY:VII-a:w5w6-2},
thus $b_4 = 1$. 
However, we also have $a_i = 1$ for some $i \in \{5,6\}$.
Then again \cref{eq:CY:VII-a:w5w6-2} implies
$b_i = a_i - 1 = 0$. A contradiction.
So we have $\{b_4, b_5, b_6\} = \{1,2,3\}$.

Inserting into \cref{eq:CY:VII-a:CYmu2} amounts to $b_1 = 1$.
Currently the degree matrix has the form
\[
	Q^0 = \begin{bmatrix*}[r]
		 1 & 1 & 0 &  -1 & -a_5 & -a_6 \\
		-1 & 0 & 1 & b_4 &  b_5 &  b_6
	\end{bmatrix*}.
\]
Recall that $a_5 = 1$ or $a_6 = 1$ holds.
So we have $b_4 > b_5$ or $b_4 > b_6$
due to the counter-clockwise orientation of $w_4, w_5, w_6$.
From this we infer $b_4 \neq 1$,
hence $b_i = 1$ for some $i \in \{5, 6\}$.
We are left with the cases $b_4 = 2,3$.
With $b_4 = 3$, inserting into \cref{eq:CY:VII-a:w5w6-2} gives
$3 a_i - 1 = b_i = 1$. A contradiction to $a_i \in \ZZ_{\geq 1}$.
With $b_4 = 2$ we deduce $a_i = 1$ from \cref{eq:CY:VII-a:w5w6-2}.
In particular $w_i = (-1, 1) = - w_1$ holds.
A contradiction to $\Eff(R)$ being pointed.

\bigskip
\noindent
\emph{Case VII-b}.
Here, we assume $\mu \in \varrho_3$.
\Cref{cor:gitconeCY} provides us with two GIT-cones
\[
	\eta_1 = \cone(w_2, w_4), \qquad
	\eta_2 = \cone(w_4, w_5).
\]
Both of them give rise to a smooth variety $X(\eta_i)$;
see \cref{CYSQMsmooth}.
Consider~$X(\eta_2)$.
Applying \cref{lem:twofacesCY} to both pairs $w_4, w_5$ and $w_4, w_6$
yields $\det(u_4, u_5) = 1$ and $\det(u_4, u_6) = 1$.
A suitable admissible coordinate change leads to
\[
	Q^0 = \begin{bmatrix*}
	 a_1 &  a_2 &  a_3 & 1 & a_5 & 0 \\
	-b_1 & -b_2 & -b_3 & 0 &   1 & 1
	\end{bmatrix*}, \quad
	a_i, b_i \in \ZZ_{\geq 1}.
\]
\Cref{lem:raypowerCY} provides us with some $k \in \ZZ_{\geq 2}$
such that $\mu = k w_3$ holds.
In particular, we have $\alpha_1 = k a_3$.
Now consider $X(\eta_1)$.
\Cref{lem:threegenerateCY} applied to the triples $w_1, w_3, w_5$
and $w_2, w_3, w_5$ shows that $a_1 + b_1a_5$ and $a_2 + b_2 a_5$
both divide $k$.
Moreover, applying \cref{lem:threegenerateCY} to $w_1, w_2, w_5$
yields $\gcd(a_1 + b_1 a_5, a_2 + b_2 a_5) = 1$.
Together with \cref{rem:weaklyCY}~(i) we obtain
\[
	(a_1 + b_1 a_5) (a_2 + b_2 a_5) a_3 \mid \alpha_1
	= a_1 + a_2 + a_3 + a_5 + 1.
\]
We expand the left-hand side and give a rough estimation:
\begin{align*}
	(a_1 + b_1 a_5) (a_2 + b_2 a_5) a_3
	&= a_1 a_2 a_3 + a_1 a_3 a_5 b_2 + a_2 a_3 a_5 b_1 + a_3 a_5^2 b_1 b_2 \\
	&\geq a_1 + a_2 + a_3 + a_5
\end{align*}
Since $\gcd(n,\, n+1) = 1$ is true for every integer $n$,
this inequation shows that equality holds in  the above divisibility condition.
From this we infer
\[
	a_2 (a_1 a_3 - 1)  + a_1(a_3 a_5 b_2 - 1) + a_3(a_2 a_5 b_1 - 1)  +  a_5 (a_3 a_5 b_1 b_2 - 1) = 1.
\]
Observe that every summand on the left-hand side is non-negative,
hence precisely one of them equals one while the other vanish.
Since $a_1, a_2, a_3, a_5$ are non-zero,
the factor in the parenthesis vanishes whenever the whole summand vanishes.
There are two summands where $b_1$ shows up in the second factor.
At least one of those parenthesis must vanish, hence $b_1 = 1$.
Repeating this argument yields $b_2 = 1$ as well as $a_3 = 1$.
Similarly, we obtain $a_1 = 1$ or $a_2 = 1$.
Altogether we have $u_3 = (1, -b_3)$ and $u_i = (1, -1)$ where $i \in \{1,2\}$.
This implies $\det(u_i, u_3) = 1 - b_3 \leq 0$;
a contradiction to our assumption that $w_1, \dotsc, w_6$ are in
counter-clockwise order.

\section{Constructing general hypersurface Cox rings}
\label{sec:veritools}
This section is devoted to the construction
of general hypersurface Cox rings with 
prescribed specifying data.
First, we describe the toolbox 
for producing general hypersurface
Cox rings from \cite{HLM}*{Sec. 4} in outline.
Then we extend it with new explicit factoriality criterions for
general hypersurface rings;
see \cref{factfromSimplexforCY,prop:ufdcritblowup}.

\begin{construction}
\label{constr:hypersurfforCY}
Consider a linear, pointed, almost free
$K$-grading on the polynomial ring 
$S := \KK[T_1, \ldots, T_r]$ and the 
quasitorus action 
$H \times \bar Z \to \bar Z$,
where 
\[
	H 
	\ := \ 
	\Spec \, \KK[K],
	\qquad\qquad
	\bar Z 
	\ := \ 
	\Spec \, S 
	\ = \ 
	\KK^r.
\]
We write $Q \colon \ZZ^r \to K$,
$e_i \mapsto w_i := \deg(T_i)$ for the degree
map.
Assume that $\Mov(S) \subseteq K_\QQ$ is of 
full dimension and fix $\tau \in \Lambda(S)$ 
with $\tau^\circ \subseteq \Mov(S)^\circ$.
Set
\[
\hat {Z}
\ := \ 
\bar Z^{ss}(\tau),
\qquad \qquad 
Z
\ := \ 
\hat {Z} \quot H.
\]
Then $Z$ is a projective toric variety with 
divisor class group $\Cl(Z) = K$ 
and Cox ring $\mathcal{R}(Z) = S$.
Moreover, fix $0 \ne \mu \in K$, and 
for $g \in S_\mu$ set
$$
R_g  := S / \langle g \rangle,
\qquad
\bar X_g  :=  V(g)  \subseteq  \bar Z,
\qquad
\hat X_g  :=  \bar X_g \cap \hat Z,
\qquad
X_g  :=  \hat X_g \quot H  \subseteq  Z.
$$
Then the factor algebra $R_g$ inherits a 
$K$-grading from $S$ and the quotient  
$X_g \subseteq Z$ is a closed subvariety.
Moreover, we have
$$ 
X_g
\ \subseteq \
Z_g
\ \subseteq \ 
Z
$$
where $Z_g \subseteq Z$ is the 
minimal ambient toric variety
of $X_g$, that means the (unique) minimal 
open toric subvariety containing $X_g$. 
\end{construction}

The spread $\mu$-homogeneous polynomials form
an open subset $U_\mu \subseteq S_\mu$.
Moreover all polynomials $g \in U_\mu$ share
the same minimal ambient toric variety $Z_g$.
We call $Z_\mu := Z_g$, where $g \in U_\mu$,
the \emph{$\mu$-minimal ambient toric variety}.
The following propositions enable us to verify smoothness of
$Z_\mu$ and the general $X_g$ in a purely combinatiorial manner.

\begin{proposition}
\label{prop:muambsmoothforCY}
In the situation of Construction~\ref{constr:hypersurfforCY}
the following are equivalent.
\begin{enumerate}
\item The $\mu$-minimal ambient toric variety $Z_\mu$ is smooth.
\item For each $\gamma_0 \preceq \gamma$ with
	$\tau^\circ \in Q(\gamma_0)^\circ$ and
	$|Q^{-1}(\mu) \cap \gamma_0| \neq 1$
	the group $K$ is generated by $Q(\gamma_0 \cap \ZZ^r)$.
\end{enumerate}
\end{proposition}

\begin{proposition} 
\label{cor:rk2bertiniforCY}
In the setting of Construction~\ref{constr:hypersurfforCY},
assume $\rk(K) = 2$ and that  $Z_{\mu} \subseteq Z$ is smooth.
If $\mu \in \tau$ holds, then $\mu$ is base point free.
Moreover, then there is a non-empty open subset of
polynomials $g \in S_\mu$ such that
$X_g$ is smooth.
\end{proposition}

\begin{remark}
\label{rem:embcox}
\label{rem:genCRirredundantforCY}
In the situation of Construction~\ref{constr:hypersurfforCY}
asumme that $R_g$ is normal, factorially graded
and $T_1, \ldots, T_r$ define pairwise
non-associated $K$-primes in $R_g$. 
Then $R_g$ is an abstract Cox ring
and we find a GIT-cone $\lambda \in \Lambda(R_g)$
with $\tau^\circ \subseteq \lambda^\circ$
and $\hat X_g = \bar X^\sstable (\lambda)$.
This brings us into
the situation of \cref{constr:mdsCY,ambtv},
so we have
\[
\Cl(X_g) \ = \ K,
\qquad\qquad
\RRR(X_g) \ = \ R_g, \qquad \qquad
\tau^\circ \subseteq \Ample(X_g).
\]
Moreover, for any $g \in U_\mu$ the variables $T_1, \ldots, T_r$ form a minimal 
system of generators for all $R_g$ if and only if
we have $\mu \ne w_i$ for $i = 1, \ldots, r$.
\end{remark}

Constructing a general hypersurface Cox ring
with prescribed specifying data essentially means to
to find a suitable open subset $U \subseteq S_\mu$
such that $R_g$, where $g \in U$, satisfies
the conditions from the above remark.
In the subsequent text we present several criterions to check
these conditions.

\begin{proposition} 
\label{prop:varclassprimeforCY}
Consider the setting of Construction~\ref{constr:hypersurfforCY}.
For $1 \le i \le r$ denote by $U_i \subseteq S_\mu$ 
the set of all $g \in S_\mu$ such that $g$ is prime 
in $S$ and $T_i$ is prime in~$R_g$.
Then $U_i \subseteq S_\mu$ is open.
Moreover, $U_i$ is non-empty if and only if 
there is a $\mu$-homogeneous prime polynomial 
not depending on $T_i$.
\end{proposition}

By a \emph{Dolgachev polytope}
we mean a convex polytope 
$\Delta \subseteq \QQ_{\geq 0}^r$ 
of dimension at least four
such that 
each coordinate hyperplane of $\QQ^r$ 
intersects $\Delta$ non-trivially
and
the dual cone of $\cone(\Delta - u; \ u \in \Delta_0)$ 
is regular for each one-dimensional face 
$\Delta_0 \preceq \Delta$.

\begin{proposition} 
\label{prop:ufdcritforCY}
In the situation of Construction~\ref{constr:hypersurfforCY},
there is a non-empty open subset of polynomials 
$g \in S_\mu$ such that the ring $R_g$ is factorial
provided that one of the following conditions 
is fulfilled:
\begin{enumerate}
\item 
$K$ is of rank at most $r-4$ and torsion free,
there is a $g \in S_\mu$ such that
$T_1, \ldots, T_r$ define primes
in $R_g$, we have $\mu \in \tau^\circ$ 
and $\mu$ is base point free on~$Z$.
\item
The set 
$\conv(\nu \in \ZZ_{\geq 0}^r; \ Q(\nu) = \mu)$
is a Dolgachev polytope.
\end{enumerate}
\end{proposition}

We combine the concepts of
algebraic modifications \cite[Sec. 4.1.2]{ADHL}
and $\Sigma$-homogenizations~\cite{HaMaWr}
to provide further factoriality criterions
for graded hypersurface rings.
These will apply to several cases where
the relation degree lies on the boundary of the
moving cone.
Let us first briefly recall the notion of
polynomials arising from Laurent polynomials 
by homogenization with respect to a lattice fan
from~\cite{HaMaWr}.

\begin{remark}
\label{reminder:Sigmadegree}
Let $\Sigma$ be a complete lattice fan in $\ZZ^n$
and $v_1, \dotsc, v_r$ the primitive lattice vectors
generating the rays of $\Sigma$.
Consider the mutually dual exact sequences
\[ \xymatrix@C=4em{
0 \ar[r] & L \ar[r] & \ZZ^r \ar[r]^{P}_{e_i \mapsto v_i} & \ZZ^n \\
0 & K \ar[l] & \ar[l]_Q \ZZ^r  & \ar[l]_{P^*} \ZZ^n & \ar[l] 0
}
\]
This induces a pointed $K$-grading on the polynomial algebra
$S = \KK[T_1, \dotsc, T_r]$ via $\deg(T_i) \coloneqq Q(e_i) \in K$.
For any $w \in K$ we denote $S_w \subseteq S$ for the finite-dimensional vector
space of homogeneous polynomials of degree $w$.

Moreover, fix a lattice polytope $B \subseteq \QQ^n$ and set
\[
	a(\Sigma) \coloneqq (a_1, \dotsc, a_r) \in \ZZ^r, \qquad
	a_i \coloneqq - \min_{u \in B} \langle u, v_i \rangle.
\]
We call $\mu = Q(a(\Sigma)) \in K$ the \emph{$\Sigma$-degree} of $B$.
\index{$\Sigma$-degree}
Besides $\mu \in K = \Cl(Z)$ regarded as a divisor class is base point free if
$\Sigma$ refines the normal fan of $B$.
The \emph{$\Sigma$-homogenization} of a
Laurent polynomial $f \in \KK[T_1^{\pm 1}, \dotsc, T_n^{\pm 1}]$
with Newton polytope $B(f)$ equal to~$B$ 
is the $\mu$-homogeneous polynomial $g = T^{a(\Sigma)} p^* f \in S$
where $p: \TT^r \rightarrow \TT^n$ is
the homomorphism of tori associated with~$P$.
Each spread polynomial $g \in S_{\mu}$ arises
as $\Sigma$-homogenization of a Laurent polynomial
\index{$\Sigma$-homogenization}
$f$ with $B(f) = B$.

Let $\Sigma_1$, $\Sigma_2$ be lattice fans
refining the normal fan $\Sigma(B)$ of $B$.
The vector space $V(B)$
of all Laurent polynomials of the form $\sum_{\nu \in B \cap \ZZ^r} a_\nu T^\nu$
fits into the following commutative diagram of vector space isomorphisms
\[
\xymatrix{
S_{\mu_1} \ar[rr]^\varphi & & S_{\mu_2} \\
& V(B) \ar[lu]^{f \mapsto T^{a(\Sigma_1)} p_1^* f} \ar[ru]_{f \mapsto T^{a(\Sigma_2)} p_2^* f} &
}
\]
Moreover, if $g \in S_{\mu_1}$ is spread,
then $\varphi(g) \in S_{\mu_2}$ is spread as well
and $g$, $\varphi(g)$ are homogenizations of a common Laurent polynomial
with respect to different fans $\Sigma_i$.
\end{remark}

We state an adapted version of \cite{ADHL}*{Thm. 4.1.2.2};
see also \cite{ADHL}*{Prop. 4.1.2.4}.

\index{$K$-factorial}

\begin{theorem}
\label{algmod}
Let $f \in \LP(n)$ be a Laurent polynomial and
$\Sigma_2 \preceq \Sigma_1 \preceq \Sigma(B(f))$
a refinement of fans in $\ZZ^n$.
Moreover,  let $g_i \in \KK[T_1, \dotsc, T_{r_i}]$ 
be the respective $\Sigma_i$-homogenization of~$f$ and
consider the $K_i$-graded algebra
\[
	R_{g_i} = \KK[T_1, \dotsc, T_{r_i}] / \langle g_i \rangle.
\]
Assume that $g_1, g_2$ are prime polynomials,
$T_1, \dotsc, T_{r_1}$ define $K_1$-primes in
$R_{g_1}$ and $T_1, \dotsc, T_{r_2}$ define $K_2$-primes in $R_{g_2}$.
Then the following statements are equivalent.

\begin{enumerate}
\item
The algebra $R_{g_1}$ is factorially $K_1$-graded.

\item 
The algebra $R_{g_2}$ is factorially $K_2$-graded.
\end{enumerate}
\end{theorem}

Now let us bring this theorem in the context of general hypersurface rings.
We observe that factoriality is inherited between general hypersurface rings
with relation degrees stemming from a common lattice polytope.

\begin{proposition}
\label{algmod-general}
Let $B \subseteq \QQ^n$ be a lattice polytope,
$\Sigma_2 \preceq \Sigma_1 \preceq \Sigma(B)$ a refinement of fans in $\ZZ^n$,
and $\mu_i \in K_i$ the respective $\Sigma_i$-degree.
Assume that for $i = 1, 2$ there is a $\mu_i$-homogeneous prime polynomial $g_i$ 
and a non-empty open subset
$U_i \subseteq S_{\mu_i}$ such that
for all $g_i \in U_i$ the variables $T_1, \dotsc, T_{r_i}$
define $K_i$-primes in the $K_i$-graded algebra
\[
	 R_{g_i} = \KK[T_1, \dotsc, T_{r_i}] / \langle g_i \rangle.
\]
Then the following statements are equivalent.

\begin{enumerate}
\item
There is a non-empty open subset of polynomials $g_1 \in S_{\mu_1}$
such that $R_{g_1}$ is $K_1$-factorial.

\item
There is a non-empty open subset of polynomials $g_2 \in S_{\mu_2}$
such that $R_{g_2}$ is $K_2$-factorial.
\end{enumerate}
\end{proposition}

\begin{proof}
We know that the subset $U_{\mu_i} \subseteq S_{\mu_i}$
of spread $\mu_i$-homogeneous polynomials
is open and non-empty.
According to \cref{reminder:Sigmadegree} there is
an isomorphism $\varphi: S_{\mu_1} \rightarrow~S_{\mu_2}$
of vector spaces
such that~$g$ and~$\varphi(g)$ arise as $\Sigma_i$-homogenization
of the same Laurent polynomial whenever $g \in U_{\mu_1}$.
Besides, by \cite{HLM}*{Lem. 4.9} the $\mu_i$-homogeneous
prime polynomials form an open subset of $S_{\mu_i}$,
which is non-empty by assumption.
Therefore, by suitably shrinking~$U_1$ and~$U_2$ we achieve
\begin{itemize}
\item
$\varphi(U_1) = U_2$, 

\item
$g_1$ and $g_2 \coloneqq \varphi(g_1)$ are respective
$\Sigma_i$-homogenizations of a common
Laurent polynomial whenever $g_1 \in U_1$,

\item
for every $g_1 \in U_1$
the ring $R_{g_1}$ is integral and
$T_1, \dotsc, T_{r_1} \in R_{g_1}$ are $K_1$-prime,

\item
for every $g_2 \in U_2$
the ring $R_{g_2}$ is integral and
$T_1, \dotsc, T_{r_2} \in R_{g_2}$ are $K_2$-prime.
\end{itemize}
In this situation \cref{algmod} tells us that
for any $g_1 \in U_1$ and $g_2 \coloneqq \varphi(g_1)$ we have
\[
	R_{g_1} \text{ is } K_1\text{-factorial}
	\; \Longleftrightarrow \;
	R_{g_2} \text{ is } K_2\text{-factorial}.
\]
Now let $V_1 \subseteq S_{\mu_1}$ be a non-empty open subset such that
$R_{g_1}$ is factorially graded for each $g_1 \in V_1$.
Then $V_2 \coloneqq \varphi(U_1 \cap V_1)$
is a non-empty open subset of $S_{\mu_2}$ and
$R_{g_2}$ is $K_2$-factorial for all $g_2 \in V_2$.
This proves ``(i) $\Rightarrow$ (ii)''.
The inverse implication is shown analogously.
\end{proof}

\begin{remark}
In the situation of Construction~\ref{constr:hypersurfforCY},
assume that $Z$ is a fake weighted projective space,
i.e., $Z$ is $\QQ$-factorial and $\Cl(Z) = K$ is of rank one.
Then $\mu \in \Cl(Z)$ ist base point free if and only if
there is an $l_i \in \ZZ_{\geq 1}$
with $\mu = l_i w_i$ for all $1 \leq i \leq 6$.
\end{remark}

According to \cite{HaMaWr}*{Rem. 5.8} general base point free hypersurfaces
in fake weighted projective spaces of dimension at least four
 always stem from Cox ring embeddings.

\begin{proposition}
\label{prop:factcritrk1}
In the situation of Construction~\ref{constr:hypersurfforCY},
suppose that $K$ is of rank one, $r \geq 5$ holds and
that for any $i = 1, \dotsc, r$ there is an $l_i \in \ZZ_{\geq 1}$
with $\mu = l_i w_i$.
Then there is a non-empty open subset of polynomials 
$g \in S_\mu$ such that the ring $R_g$ is normal and $K$-factorial,
and $T_1, \dotsc, T_r \in R_g$ are prime.
\end{proposition}

\begin{corollary}
\label{factfromSimplexforCY}
Let $n \geq 4$, $B \subseteq \QQ^n$ an integral $n$-simplex,
$\Sigma$ a fan in $\ZZ^n$ refining the normal fan of $B$,
and $\mu \in K$ the $\Sigma$-degree of $B$.
Assume that there is a $\mu$-homogeneous prime polynomial $g$ 
and a non-empty open subset
$U \subseteq S_{\mu}$ such that
for all $g \in U$ the variables $T_1, \dotsc, T_{r}$
define $K$-primes in the $K$-graded algebra
\[
	 R_{g} = \KK[T_1, \dotsc, T_{r}] / \langle g \rangle.
\]
Then there is a non-empty open subset of polynomials $g \in S_{\mu}$
such that $R_{g}$ is $K$-factorial.
\end{corollary}

\begin{proof}
Since $B$ a is simplex,
the toric variety associatd with $\Sigma(B)$ is
a fake weighted projective space.
Now we apply \cref{algmod-general} to the
refinement $\Sigma \preceq \Sigma(B)$ and the
suitable open subset of polynomials provided by \cref{prop:factcritrk1}.
\end{proof}

In many situations we encouter it can be read of straight from the 
specifying data whether the conditions from \cref{factfromSimplexforCY} are met.

\begin{corollary}
\label{prop:ufdcritblowup}
Situation as in Construction~\ref{constr:hypersurfforCY}.
Assume that we have $r \geq 5$, $K = \ZZ^2$ and 
the degree matrix is of the form
\[
	Q = [w_1, \dotsc, w_{r+1}] = 
	\begin{bmatrix*}[r]
		 x_1 & \hdots &  x_{r} & 0 \\
		-d_1 & \hdots & -d_{r} & 1
	\end{bmatrix*}, \qquad
	x_i \in \ZZ_{\geq 1},\;
	d_i \in \ZZ_{\geq 0}.
\]
Then for any $\mu = (\mu_1, \mu_2) \in K = \ZZ^2$ satisfying the subsequent conditions
there is a non-empty open subset of polynomials 
$g \in S_\mu$ such that $R_g$ is factorial:
\begin{enumerate}
\item
for each $i$ there exists some $l_i \in \ZZ_{\geq 1}$ with $\mu = l_i x_i$,

\item
$\mu_2 = - \min_{\nu} \nu_1 d_1 + \dotsb + d_{r} \nu_{r}$
where the minimum runs over all lattice points
 $\nu = (\nu_1, \dotsc, \nu_r) \in \ZZ^r_{\geq 0}$
with $\nu_1 x_1 + \dotsb + \nu_{r} x_{r} = \mu_1$,

\item
there is some $g \in S_\mu$ such that $T_1, \dotsc, T_{r+1}$ define primes in $R_g$.
\end{enumerate}
\end{corollary}

\begin{proof}
Observe that each $r-1$ of $x_1, \dotsc, x_r$ generate $\ZZ$ as a group since
the first coordinate of $w_{r+1}$ vanishes and
the $\ZZ^2$-grading associated with $Q$ is almost free according to
the assumptions made in \cref{constr:hypersurfforCY}.
Consider the weighted projective space
\[
	Z' \coloneqq \PP(x_1, \dotsc, x_r).
\]
Condition~(i) ensures that $\mu_1 \in \ZZ = \Cl(Z')$ regarded as a divisor class
on~$Z'$ is ample and base point free.
Choose some representative $D \in \WDiv(Z')$ of $\mu_1$.
The associated divisorial polytope $B \coloneqq B(D) \subseteq \QQ^{r-1}$
is a full-dimensional integral simplex.

The normal fan $\Sigma'$ of $B$ is a lattice fan in $\ZZ^{r-1}$ corresponding with $Z'$.
Write $v_1, \dotsc, v_r \in \ZZ^{r-1}$ for the primitive ray generators of $\Sigma'$.
Observe that the maps
\[
	P': \ZZ^r \rightarrow \ZZ^{r-1},\; e_i \mapsto v_i, \qquad \qquad
	Q': \ZZ^r \rightarrow \ZZ,\; e_i \mapsto x_i
\]
fit into a mutually dual pair of exact
sequences as shown in \cref{reminder:Sigmadegree}.
Now set
\[
	v_{r+1} \coloneqq d_1 v_1 + \dotsb + d_r v_r \in \ZZ^{r-1}, \qquad
	d \coloneqq (d_1, \dotsc, d_r) \in \ZZ^r.
\]
The second row of $Q$ encodes the relation statisfied by
$v_1, \dotsc, v_{r+1}$ thus
indicates that the following maps
constitute a pair of mutually dual sequences as well
\[
\begin{array}{ll}
	P: \ZZ^{r+1} \rightarrow \ZZ^n,\; e_i \mapsto v_i, &
	Q: \ZZ^{r+1} \rightarrow \ZZ^2,\; e_i \mapsto w_i.
\end{array}
\]
Since the first~$r$ columns of~$Q$ generate~$\ZZ^2$,
the vector $v_{r+1} \in \ZZ^{r-1}$ is primitive;
see \cite{ADHL}*{Lemma 2.1.4.1}.
This allows us to consider the stellar subdivision $\Sigma$ of $\Sigma'$
along~$v_{r+1}$.

We show that $\mu \in \ZZ^2$ is the $\Sigma$-degree $\mu_B$ of $B$.
First note that $\mu_1$ is the $\Sigma'$-degree of~$B$ by construction.
Consider
\[
	a' \coloneqq a(\Sigma') = (a_1', \dotsc, a_r'), \qquad
	a \coloneqq a(\Sigma) = (a_1, \dotsc, a_{r+1})
\]
from \cref{reminder:Sigmadegree}.
Since~$\Sigma$ arises from $\Sigma'$ by introducing an $(r+1)$-th ray,
we have $a_i = a_i'$ for $i = 1, \dotsc, r$.
From this we infer
\[
	\mu_1 = Q'(a') = a_1 x_1 + \dotsb + a_r x_r, \qquad
	\mu_B = Q(a) = a_1 w_1 + \dotsb + a_{r+1} w_{r+1}.
\]
As the first coordinate of $w_{r+1}$ vanishes,
we conclude that the first coordinate of~$\mu_B$
equals~$\mu_1$.
It remains to investigate the second coordinate of $\mu_B$.
We have
\[
	a_{r+1}
	= - \min_{u \in B} \langle u,\, v_{r+1} \rangle
	= - \min_{u \in B} \langle u,\, P'(d) \rangle 
	= - \min_{u \in B} \langle (P')^*u,\, d \rangle.
\]
Using this presentation of $a_{r+1}$,
the second coordinate of $\mu_B$ is given as
\[
	a_{r+1} - \sum_{i = 1}^r a_i d_i
	= - \min_{u \in B} \langle (P')^*u,\, d \rangle - \langle a',\, d \rangle
	= - \min_{u \in B} \langle (P')^*u + a',\, d \rangle.
\]
From condition~(ii) and the fact
that the lattice points $\nu \in \ZZ^r_{\geq 0}$ with $Q'(\nu) = \mu_1$
are precisely those of the form $\nu = (P')^*u + a'$,
where $u \in B \cap \ZZ^{r-1}$,
follows that the second coordinate of $\mu_B$ equals $\mu_2$.
Altogether we have verified $\mu = \mu_B$.

Note that condition (iii) ensures the existence of a $\mu$-homogeneous
prime polynomial; see \cref{prop:varclassprimeforCY}.
The above discussion combined with condition~(iii)
ensures that we may apply \cref{factfromSimplexforCY}
to $B$ and $\Sigma$
which finishes the proof.
\end{proof}

\section{Proof of Theorem~\ref{thm:CY3folds}: Verification}

The second mission in the proof of \cref{thm:CY3folds}
is to ensure that the list of specifying data given there
does not contain any superfluous items.
So we have to verify that all items from \cref{thm:CY3folds} are realized by
pairwise non-isomorphic smooth Calabi-Yau threefolds having a
(general) hypersurface Cox ring.

\medskip
Let us highlight generator and relation degrees
of a graded algebra as 
invariants of hypersurface Cox rings
that distinguish varieties with different specifying data;
see \cite{HaKeWo}*{Sec. 2} for details.

\begin{remark}
\label{rem:HFinvars}
Let $R = \bigoplus_{w \in K} R_w$ be an integral pointed $K$-graded algebra.
We denote $S(R) = \{w \in K;\, R_w \neq 0\}$.
An important invariant of $R$
is the \emph{set of generator degrees}
\[
	\Omega_R \coloneqq
	\left\{ w \in S(R);\, R_w \nsubseteq R_{< w} \right\} \subseteq K
\]
where $R_{< w}$ denotes the subalgebra of $R$ spanned by all homogeneous
componentes $R_{w'}$ such that $w = w' + w_0$ holds for some $0 \neq w_0 \in S(R)$.
In the situation of \cref{setting:rho2forCY}
the set of generator degrees
is given as
\[
	\Omega_{R} = \{w_1, \dotsc, w_r\} \subseteq K.
\]
The set of generator degrees is unique
and does not depend on a graded presentation of~$R$.
From this emerges another invariant:
Choose pairwise different $u_1, \dotsc, u_l \in K$ such that
$\Omega_R = \{u_1, \dotsc, u_l\}$
and set $d_i \coloneqq \dim_\KK R_{u_i}$.
By suitably reordering $u_1, \dotsc, u_l$ we achieve
$d_1 \leq \dotsc \leq d_l$.
We call $(d_1, \dotsc, d_l)$
the \emph{generator degree dimension tuple} of $R$.
\index{generator degree dimension tuple}
If two graded algebras are isomorphic,
then they have the same generator degree dimension tuple.

Moreover, if $R$ admits an irredundant graded presentation
$R = \KK[T_1, \dotsc, T_r] / \langle g \rangle$,
then the \emph{relation degree} $\mu = \deg(g) \in K$ is unique
\index{relation degree}
and does not depend on the choice of the minimal graded presentation.
\end{remark}

\begin{lemma}
\label{lem:reldeg}
Consider $n$-dimensional varieties $X_1$, $X_2$ with hypersurface
Cox rings having relation degree $\mu_1$ resp.\ $\mu_2$.
If $X_1$ and $X_2$ are isomorphic, then $\mu_1^n = \mu_2^n$ where~$\mu_i^n$
is the self-intersection number of $\mu_i$
regarded as a divisor class on $X_i$.
\end{lemma}

\begin{proof}
Let $\varphi: X_1 \rightarrow X_2$ be an isomorphism.
Then the induced pull-back maps
\[
	\varphi^*: \mathcal{R}(X_2) \rightarrow \mathcal{R}(X_1), \qquad
	\tilde{\varphi}^*: \Cl(X_2) \rightarrow \Cl(X_1)
\]
form an isomorphism $(\varphi^*, \tilde{\varphi}^*)$
of $\Cl(X_i)$-graded algebras.
From this we deduce that the pull-back $\tilde{\varphi}^*(\mu_2) \in \Cl(X_2)$
of the relation degree $\mu_2 \in \Cl(X_2)$ of $\mathcal{R}(X_2)$
is the unique relation degree $\mu_1 \in \Cl(X_1)$ of $\mathcal{R}(X_1)$;
see also \cref{rem:HFinvars}.
Hence \mbox{$\mu_1^n = \tilde{\varphi}^*(\mu_2)^n = \mu_2^n$}.
\end{proof}

\begin{proof}[Proof of \cref{thm:CY3folds}: Verification]
We show that each item from \cref{thm:CY3folds} indeed stems from
a smooth Calabi-Yau threefold with a general hypersurface Cox ring.

Let $(Q, \mu, u)$ be specifying data as presented in \cref{thm:CY3folds}.
Consider the linear $K$-grading on
$S = \KK[T_1, \dotsc, T_6]$
given by $Q: \ZZ^6 \rightarrow K$.
We run Construction~\ref{constr:hypersurfforCY}
with the unique GIT-chamber $\tau \in \Lambda(S)$
containing $u$ in its relative interior $\tau^\circ$.
In doing so $u \in \Mov(S)^\circ$ guarantees
$\tau^\circ \subseteq \Mov(S)^\circ$.
In what follows we construct a non-empty
open subset $U \subseteq U_\mu$ of polynomials satisfying the
conditions from \cref{rem:embcox},
thereby obtaining
a smooth general Calabi-Yau hypersurface Cox ring.
This is done by starting with $U \coloneqq U_\mu$
and shrinking $U$ successively.

Since $\mu \neq w_i$ holds for all $i$,
Remark~\ref{rem:genCRirredundantforCY} ensures
that $T_1, \dotsc, T_6$ form a minimal system 
of generators for $R_g$, whenever $g \in U_\mu$.
We want to achieve $K$-primeness of $T_1, \dotsc, T_6 \in R$.
Here Numbers~\ref{CY3fold:I-tors1} and~\ref{CY3fold:III-i-3-tors}
have to be treated separately.
For all remaining items from \cref{thm:CY3folds} and any $1 \leq i \leq 6$
we find in \cref{tab:CY-primebin} on page~\pageref{tab:CY-primebin}
a $\mu$-homogeneous prime binomial
$T^{\kappa} - T^{\nu} \in S$ not
depending on $T_i$.
Thus, \cref{prop:varclassprimeforCY} allows us
to shrink $U$ such that $T_1, \dotsc, T_6$ define primes in $R_g$
for all $g \in U$.

\medskip
\noindent
\emph{Number~\ref{CY3fold:I-tors1}.}
For Number~2 observe that all the generator degrees
 $w_i = \deg(T_i)$ are indecomposable in the weight monoid
\[
	S(R)
	= \{ u \in K;\; R_u \neq 0 \}
	= \operatorname{Pos}_\ZZ(w_1, \dotsc, w_6)
	\subseteq K.
\]
Thus every $T_i \in R_g$ is $K$-irreducible.
As soon as we know that $R_g$ is $K$-factorial,
we may conclude that $T_i$ is $K$-prime.

\medskip
\noindent
\emph{Number~\ref{CY3fold:III-i-3-tors}.}
\Cref{tab:CY-primebin} on page~\pageref{tab:CY-primebin}
shows $\mu$-homogeneous prime binomials
$T^{\kappa} - T^{\nu} \in S$ not
depending on $T_i$
for $i = 1, \dotsc, 5$.
Thus, \cref{prop:varclassprimeforCY} allows us
to shrink $U$
such that $T_1, \dotsc, T_5$ define primes in $R_g$
for all $g \in U$.

Observe that $T_6$ defines a $K$-prime in $R_g$ if and only if
$h \coloneqq g(T_1, \dotsc, T_5, 0) \in S$ is $K$-prime.
Since $S$ is a UFD, thus $K$-factorial, the latter is equivalent
to $h \in S$ being $K$-irreducible.
The only monomials of degree $\mu$ not depending on $T_6$
are $T_4^3$ and $T_5^3$, hence $h = aT_4^3 - b T_5^3$.
Note that $T_4^3$, $T_5^3$ are vertices of the polytope
\[
	\conv\left(\nu \in \ZZ_{\geq 0}^6;\, \deg(T^\nu) = \nu\right).
\]
From $g$ being spread we infer $a, b \in \KK^*$.
For degree reasons, any non-trivial factorization of $h$
has a linear form $\ell = a'T_4 + b'T_5$ with $a', b' \in \KK^*$
among the factors.
From $w_4 \neq w_5$ we deduce that
such $\ell$ is not homogeneous w.r.t the $K$-grading.
We conclude that
$h$ admits no non-trivial presentation as product of homogeneous elements,
i.e., $h \in S$ is $K$-irreducible.
This implies that $T_6 \in R_g$ is $K$-prime.

\medskip
\noindent
We take the next step,
that is to make sure that
each $R_g$ is normal and factorially graded.
For example this holds when $R_g$ admits unique factorization.
Whenever $K$ is torsion-free the converse
is also true.
Here we encounter different classes of candidates.

\medskip
\noindent
\emph{Numbers 1, 2, 5, 6, 10 -- 22, and 26 -- 28}.
One directly checks that the
convex hull over the $\nu \in \ZZ_{\geq 0}^6$ 
with $Q(\nu) = \mu$ is Dolgachev polytope;
we have used the Magma function \texttt{IsDolgachevPolytope} 
from \cite{MagmaToolkit} for this task.
Proposition~\ref{prop:ufdcritforCY}~(ii) ensures
that $R_g$ is factorial after
suitably shrinking $U$.

\medskip
\noindent
\emph{Numbers
\ref{CY3fold:II-1},
\ref{CY3fold:II-2}, and
\ref{CY3fold:VI-iii-1}.}
Here, the cone $\tau' = \cone(w_3) \in \Lambda(S)$
satisfies $(\tau')^\circ \subseteq \Mov(S)^\circ$. 
Thus, Construction~\ref{constr:hypersurfforCY} gives
raise to a toric variety $Z'$.
We have $\mu \in (\tau')^\circ$ and
one directly verifies that $\mu$ is 
basepoint free for $Z'$.
Hence Proposition~\ref{prop:ufdcritforCY}~(i) shows
that after shrinking $U$ suitably,
$R_g$ admits unique factorization 
for all $g \in U$.

\medskip
\noindent
\emph{Number~\ref{CY3fold:III-i-3-tors}.}
We are aiming to apply \cref{factfromSimplexforCY}.
For this purpose we have to verify that $\mu$ occurs as degree associated with
a simplex in the sense of \cref{reminder:Sigmadegree}.
The following polytope does the job:
\[
	B = \conv( (0,0,0,0),\, (0,0,0,3)\, (0,0,9,-3),\, (3,0,3,-1),\, (3,3,3,-2) )
	\subseteq \mathbb{Q}^4.
\]
The rays of its normal fan $\Sigma(B)$ are given as the columns of the
following matrix
\[
 	P_1 = \begin{bmatrix*}[r]
 	-2 & 0 & -1 & 0 &  1 \\
 	-1 & 1 &  0 & 1 & -1 \\
 	-2 & 1 &  1 & 0 &  0 \\
 	-3 & 3 &  0 & 0 &  0 
 	\end{bmatrix*}.
\]
Now consider the stellar subdivision $\Sigma_2$ of $\Sigma(B)$ along
$(-1, 0, 0, 0)$. The associated data of~$\Sigma_2$ is $K_2 = \ZZ^2 \times \ZZ / 3 \ZZ$
and
\[
	P_2 = \begin{bmatrix*}[r]
	-2 & 0 & -1 & 0 &  1 & -1 \\
	-1 & 1 &  0 & 1 & -1 &  0 \\
	-2 & 1 &  1 & 0 &  0 &  0\\
	-3 & 3 &  0 & 0 &  0 &  0 
	\end{bmatrix*}, \qquad
	Q_2 = \begin{bmatrix*}[r]
	1 & 1 & 1 & 0 & 0 & -3 \\
	0 & 0 & 0 & 1 & 1 &  1 \\
	\overline{0} & \overline{1} & \overline{2} & \overline{1} & \overline{2} & \overline{0}
	\end{bmatrix*}.
\]
We compute the $\Sigma_2$-degree $\mu_2$ of $B$. Observe
$a(\Sigma_2) = (9, 0, 0, 0, 3)$.
From this we infer $\mu_2 = Q_2(a(\Sigma_2)) = (0, 3, \overline{0})$.
Note that $(Q_2, \mu_2)$ coincides with the specifying data $(Q, \mu)$ for
which we run the verification process.
In the previous step of this process we have ensured that $U \subseteq S_\mu$ is
a non-empty open subset of prime polynomials such that $T_1, \dotsc, T_6$ define $K$-primes
in $R_g$ whenever $g \in U$.
According to \cref{factfromSimplexforCY}
we may shrink $U$ such that $R_g$ is $K$-factorial for each $g \in U$.

Finally, Bechtold's criterion \citelist{\cite{Be}*{Cor. 0.6} \cite{HaHiWr}*{Prop. 4.1}}
directly implies that $R_g$ is normal
since each five of $w_1, \dotsc, w_6$ generate~$K$ as a group.

\medskip
\noindent
\emph{Numbers
\ref{CY3fold:III-i-5},
\ref{CY3fold:III-ii-1},
\ref{CY3fold:III-ii-2}, and 
\ref{CY3fold:V-ii-1},
\ref{CY3fold:V-ii-2} .}
By applying a suitable 
coordinate change
we achieve that the degree matrix $Q$ and
the relation degree $\mu$ are as in the following table.

\begin{center}
\small
\begin{tabular}{lcc}
\toprule
\emph{No.} & $Q$ & $\mu$ \\ 
\midrule
\ref{CY3fold:III-i-5} & {\small $
	\begin{bmatrix*}[r]
	 1 &  1 &  1 &  6 &  9 & 0 \\
	-1 & -1 & -1 & -4 & -6 & 1	
	\end{bmatrix*}$} &
	$(18, -12)$ \\ \midrule

\ref{CY3fold:III-ii-1} & {\small $
	\begin{bmatrix*}[r]
	 1 &  1 &  1 &  6 &  9 & 0 \\
	-1 & -1 & -1 & -4 & -6 & 1	
	\end{bmatrix*}$} &
	$(8, -4)$ \\ \midrule

\ref{CY3fold:III-ii-2} & {\small $
	\begin{bmatrix*}[r]
	0 &  2 &  2 &  2 &  1 &  1 \\
	1 & -1 & -1 & -1 & -1 & -1	
	\end{bmatrix*}$} &
	$(8, -4)$ \\ \midrule

\ref{CY3fold:V-ii-1} & {\small $
	\begin{bmatrix*}[r]
	0 &  2 &  2 &  4 &  3 &  1 \\
	1 & -1 & -1 & -2 & -2 & -1	
	\end{bmatrix*}$} &
	$(12, -6)$ \\ \midrule

\ref{CY3fold:V-ii-2} & {\small $
	\begin{bmatrix*}[r]
	0 &  2 &  2 &  2 &  7 &  1 \\
	1 & -1 & -1 & -1 & -4 & -1	
	\end{bmatrix*}$} &
	$(14, -7)$ \\
\bottomrule	
\end{tabular}
\end{center}
We apply \cref{prop:ufdcritblowup}.
In the last three cases it is necessary to reorder the variables such
that $Q$ has precisely the shape requested by \cref{prop:ufdcritblowup}.
Now the conditions from there can be directly checked.
As a result, we may shrink $U$ such that
each $R_g$ is a factorial ring.

\medskip
\noindent
\emph{Number~\ref{CY3fold:V-ii-3}.}
Again we want to use \cref{factfromSimplexforCY}
thus we have to present $\mu$ as degree associated with
a simplex in the sense of \cref{reminder:Sigmadegree}.
Consider
\[
	B = \conv((0,0,0,0),\, (0,0,0,8)\, (0,8,0,0),\, (0,0,4,0),\, (2,2,1,2))
	\subseteq \mathbb{Q}^4.
\]
Its normal fan $\Sigma_1 = \Sigma(B)$ has the rays 
given by the columns of the matrix
\[
	P_1 = \begin{bmatrix*}[r]
	0 & 0 & 1 & -1 & 3 \\
	1 & 0 & 2 & -1 & 1 \\
	0 & 1 & 2 & -1 & 1 \\
	0 & 0 & 3 & -1 & 1
	\end{bmatrix*}.
\]
Now consider the stellar subdivision $\Sigma_2$
of $\Sigma(B)$ along $(1, 0, 0, 0)$.
Here associated data of~$\Sigma_2$ is given by
 $K_2 = \ZZ^2$ and
\[
	P_2 = \begin{bmatrix*}[r]
	1 & 0 & 0 & 1 & -1 & 3 \\
	0 & 1 & 0 & 2 & -1 & 1 \\
	0 & 0 & 1 & 2 & -1 & 1 \\
	0 & 0 & 0 & 3 & -1 & 1
	\end{bmatrix*}, \qquad
	Q_2 = \begin{bmatrix*}[r]
	 2 & 1 & 1 & 1 & 3 & 0 \\
	-2 & 0 & 0 & 0 & 1 & 1 \\
	\end{bmatrix*}.
\]
We compute the $\Sigma_2$-degrees $\mu_2$ of $B$. Observe
$a(\Sigma_2) = (0, 8, 0, 0, 0)$.
From this we infer 
$\mu_2 = Q_2(a(\Sigma_2)) = (8, 0)$.
Here $(Q_2, \mu_2)$ equals $(Q, \mu)$ from the specifying data for
which we run the verification process.
In the previous step of this process we have ensured that $U \subseteq S_\mu$ is
a non-empty open subset such that $T_1, \dotsc, T_6$ define primes
in $R_g$ whenver $g \in U$.
Now \cref{factfromSimplexforCY} shows that
we may shrink $U$ such that $R_g$ is factorial for each $g \in U$.

\medskip
\noindent
At this point we have that $U$ defines a general hypersurface Cox ring.
Note that \cref{prop:anticanclassCY} immediately
yields that the corresponding varieties $X_g$ are weakly Calabi-Yau.
The next step is
to attain~$X_g$ being smooth.
Checking the condition from \cref{prop:muambsmoothforCY} with
the help of the Magma program \texttt{IsMuAmbientSmooth} from~\cite{MagmaToolkit}
shows that $Z_\mu$ is smooth in all 30 cases.
Observe that we have $\mu \in \tau$ except
for Numbers~\ref{CY3fold:III-iii-1b},
\ref{CY3fold:IV-ii-2b},
\ref{CY3fold:IV-ii-3b},
and \ref{CY3fold:VI-i-1a}.
Whenever $\mu \in \tau$ holds we may apply \cref{cor:rk2bertiniforCY}
allowing us to shrink $U$ once more such that 
$X_g$ is smooth for all $g \in U$.
The four exceptional cases turn out to be small quasimodifications
of smooth weakly Calabi-Yau threefolds, hence are smooth
by \cref{CYSQMsmooth}.
Eventually \cref{rem:weaklyCY}~(ii) ensures that $X_g$ is
Calabi-Yau.

\bigskip
\noindent
The last task in the proof of \cref{thm:CY3folds} is to make sure that
two varieties from different
families from \cref{thm:CY3folds} are non-isomorphic.
Note that if two varieties from \cref{thm:CY3folds} are isomorphic,
then their Cox rings are isomorphic as graded rings.
For each family from \cref{thm:CY3folds}
we give the number~$l$ of generator degrees,
the entries of the generator degree dimension tuple $(d_1, \dotsc, d_l)$ and
the self-intersection number~$\mu^3$ of the relation degree
in the following table.

\begin{center}
\footnotesize
\begin{minipage}[b]{0.48\linewidth}
\begin{filecontents*}{CY-invars.csv}
n;d1;d2;d3;d4;d5;d6;muins
2;3;3;--;--;--;--;486
6;1;1;1;1;1;1;162
3;2;2;6;--;--;--;512
4;2;2;5;31;--;--;864
3;1;3;5;--;--;--;513
6;1;1;1;1;2;3;243
3;1;3;8;--;--;--;594
4;1;3;29;66;--;--;1944
3;1;2;6;--;--;--;512
4;1;2;5;31;--;--;864
3;2;3;7;--;--;--;513
3;2;3;7;--;--;--;512
3;2;3;31;--;--;--;864
4;1;2;4;31;--;--;864
4;1;2;4;8;--;--;520
4;1;2;4;8;--;--;512
4;1;2;5;9;--;--;539
4;1;2;5;9;--;--;512
4;1;3;4;10;--;--;567
4;1;3;4;32;--;--;896
4;1;3;7;35;--;--;992
5;1;2;3;4;28;--;784
5;1;2;4;7;32;--;912
5;1;1;3;4;8;--;432
4;1;1;4;21;--;--;686
4;1;1;3;14;--;--;512
5;1;2;3;6;31;--;864
5;1;2;3;6;31;--;872
5;1;1;3;4;29;--;808
5;1;1;3;4;4;--;432
\end{filecontents*}
\csvreader[%
separator=semicolon,%
longtable={rcccccccc},%
table head=\toprule,%
after head={%
\emph{No.} & $l$ & $d_1$ & $d_2$ & $d_3$ & $d_4$ & $d_5$ & $d_6$ & $\mu^3$ \\ \midrule%
},%
table foot=\bottomrule,%
filter test=\ifnumless{\thecsvinputline}{17}%
]{CY-invars.csv}{}%
{\thecsvrow & \csvcoli & \csvcolii & \csvcoliii & \csvcoliv & \csvcolv & \csvcolvi & \csvcolvii & \csvcolviii}
\end{minipage}
\begin{minipage}[b]{0.48\linewidth}
\csvreader[%
separator=semicolon,%
longtable={rccccccc},%
before first line=\addtocounter{csvrow}{15},%
table head=\toprule,%
after head={%
\emph{No.} & $l$ & $d_1$ & $d_2$ & $d_3$ & $d_4$ & $d_5$ & $\mu^3$ \\ \midrule%
},%
table foot=\bottomrule,%
filter test=\ifnumgreater{\thecsvinputline}{16}%
]{CY-invars.csv}{}%
{\thecsvrow & \csvcoli & \csvcolii & \csvcoliii & \csvcoliv & \csvcolv & \csvcolvi & \csvcolviii}
\end{minipage}
\end{center}

Most of the varieties from \cref{thm:CY3folds} are distinguished
by the generator degree dimension tuple.
Note that the pairs
having the same generator dimension degree tuple
are precisely Numbers~\ref{CY3fold:III-iii-1a} \&~\ref{CY3fold:III-iii-1b},
\ref{CY3fold:IV-ii-2a}~\&~\ref{CY3fold:IV-ii-2b},
\ref{CY3fold:IV-ii-3a}~\&~\ref{CY3fold:IV-ii-3b} and
\ref{CY3fold:VI-i-1a}~\&~\ref{CY3fold:VI-i-1b}
as they share the same Cox ring.
These pairs can be distinguished by the relation degree
self-intersection number; see \cref{lem:reldeg}.
\end{proof}

\begin{landscape}
\scriptsize
\setcounter{table}{0}
\begin{longtable}{crrrrrr}
\captionsetup{font=scriptsize}
\caption{%
Binomials used to ensure primeness
of $T_1, \dotsc, T_6 \in R_g$ in the proof of \cref{thm:CY3folds}%
}
\label{tab:CY-primebin} \\
\toprule
\emph{No.} & $T_1$ & $T_2$ & $T_3$ & $T_4$ & $T_5$ & $T_6$
\csvreader[separator=semicolon,head to column names,before line={\\ \addlinespace[2pt]},before first line=\\ \midrule,late after last line=\\ \bottomrule]{CY-verification_binomials.csv}{}%
{ \csvcoli & $\csvcolii$ & $\csvcoliii$ & $\csvcoliv$ & $\csvcolv$ & $\csvcolvi$ & $\csvcolvii$ }
\end{longtable}
\end{landscape}

\end{document}